%% file: SinkhornInformationGeometry.tex
\documentclass[11pt]{amsart}

\newif\iflong
\longfalse

\usepackage[T1]{fontenc}
\usepackage{lmodern}
\usepackage[a4paper,top=2.6cm,bottom=2.6cm,left=2.7cm,right=2.7cm]{geometry}
\usepackage{microtype}
\usepackage{amsmath,amssymb,amsfonts}
\usepackage{mathrsfs}
\usepackage{graphicx}
\usepackage{xcolor}
\usepackage{booktabs}
\usepackage[numbers,sort&compress]{natbib}
\usepackage[colorlinks=true,linkcolor=blue!50!black,citecolor=blue!50!black,urlcolor=blue!50!black]{hyperref}

\makeatletter
\renewcommand\subsection{\@startsection{subsection}{2}\z@{.7\linespacing\@plus\linespacing}{.5\linespacing}{\normalfont\bfseries}}
\makeatother

\theoremstyle{plain}
\newtheorem{theorem}{Theorem}[section]
\newtheorem{proposition}[theorem]{Proposition}
\newtheorem{lemma}[theorem]{Lemma}
\newtheorem{corollary}[theorem]{Corollary}

\theoremstyle{remark}
\newtheorem{example}[theorem]{Example}
\newtheorem{remark}[theorem]{Remark}

\theoremstyle{definition}
\newtheorem{definition}[theorem]{Definition}
\newtheorem{assumption}{Assumption}

\newcommand{\eps}{\varepsilon}
\newcommand{\PX}{\mathcal{P}(X)}
\newcommand{\MX}{\mathcal{M}(X)}
\newcommand{\MzX}{\mathcal{M}_{0}(X)}
\newcommand{\MpX}{\mathcal{M}_{+}(X)}
\newcommand{\CX}{C(X)}
\newcommand{\Hc}{H_{c}}
\newcommand{\Kmu}{K_{\mu}}
\newcommand{\Hmu}{H_{\mu}}
\newcommand{\fmu}{f_{\mu}}
\newcommand{\amu}{a_{\mu}}

\newcommand{\OTe}{\mathrm{OT}_{\eps}}
\newcommand{\Se}{S_{\eps}}
\newcommand{\KL}{\mathrm{KL}}

\DeclareMathOperator{\id}{id}

\begin{document}

\title[A Hilbert Sinkhorn Manifold]{The information geometry of the Sinkhorn divergence on a Hilbert manifold}

\author{Paul Navas}
\address{Department of Statistics, TU Dortmund University, Vogelpothsweg 78, 44221 Dortmund, Germany}
\email{paul.navas@tu-dortmund.de}
\urladdr{\url{https://orcid.org/0009-0005-6885-2108}}

%\dedicatory{\hfill Dedicado a MNR}


\begin{abstract}
The Sinkhorn divergence, the debiased entropic optimal transport cost, determines a metric tensor on probability measures defined along curves, due to Lavenant et al. We construct a real-analytic Hilbert manifold of measures, the smooth stratum, on which this tensor is a weak Riemannian metric, and we continue Eguchi's construction of a geometry from a divergence to the third order. The stratum consists of the measures with a positive Sobolev density on a closed Riemannian manifold, and the strata form a scale whose intersection is the Fr\'echet manifold of smooth positive densities of Ay. Following it, we characterize the metric through the gradients of mean values: its cotangent form is a spectral reweighting of the Dirichlet form of a Markov chain, so the metric is of the same kind as the Fisher-Rao and the Otto metric, and on the stratum it is bounded by a multiple of the Fisher-Rao metric (but not conversely). At third order, the Sinkhorn divergence is self-dual: both connection forms of Eguchi's construction coincide with the Levi-Civita form of the metric, and the cubic tensor vanishes, as a consequence of the transposition symmetry of the transport problem. On every stratum, the Levi-Civita form is represented, so the stratum has a real-analytic Levi-Civita connection, which on the Fr\'echet manifold is the connection of Ay and Schwachh\"ofer.
\end{abstract}

\maketitle

\input{sec1_intro}
\input{sec2}
\input{sec3}
\input{sec4}
\input{sec5}
\input{sec6}

\section*{Use of generative AI}
The author used Claude (Anthropic) to restructure, edit and polish the text. Most of the mathematical results of this work were derived when AI was still saying that 1+1=11. Nonetheless, it must be acknowledged that its assistance significantly aided in reducing the length and improving the quality of the presentation. Claude also proposed the argument behind Proposition~\ref{prop:nothessian} and checked it numerically, drafted the comparison with the Fisher-Rao metric on the stratum (Remark~\ref{rem:frsinkhorn}), the improved constants in Propositions~\ref{prop:operators} and~\ref{prop:sandwich} (the former from a Doeblin minorization that the author proposed following \citet[Proposition~19]{Vialard2019EntropicNotes}) and the expansions in Remark~\ref{rem:epslimits}, and assisted in writing up, restructuring and shortening the author's argument for the Levi-Civita connection on the stratum (Lemmas~\ref{lem:sobolevK} and~\ref{lem:potanalytic}, Proposition~\ref{prop:range}, Theorem~\ref{thm:levicivita}). The author checked all statements and proofs and takes full responsibility for the content.

\section*{Competing interests}
The author has no competing interests to declare that are relevant to the content of this article.

\appendix
\input{app_proofs}
\input{app_thirdvar}

\bibliographystyle{plainnat}
\bibliography{bibliography_sink}

\end{document}

%% file: sec1_intro.tex
\section{Introduction}\label{sec:intro}

On a manifold of probability measures, a divergence, that is, a smooth nonnegative function of two measures that vanishes only on the diagonal and whose second-order term in the expansion around the diagonal is positive definite, determines a geometry by the construction of \citet{Eguchi_1983,Eguchi_1992}: that second-order term is a metric tensor, and the third-order terms are the first-kind forms of a pair of affine connections, dual to each other with respect to the metric, whose difference is a cubic tensor. For the Kullback-Leibler divergence these are the Fisher-Rao metric, the exponential and the mixture connection, and a multiple of the Amari-Chentsov tensor \citep{AyJostLeSchwachhoefer2017}; for the squared Wasserstein distance, the second-order term is the Otto metric, which \citet{Ay2024Otto} turns into a dual structure, the exponential connection coming from the dual of the mixture connection, together with its canonical contrast function. The Sinkhorn divergence of \citet{Genevay2018SinkhornDivergences} and \citet{feydy_sejourné_vialard_amari_trouvé_peyré_interpolating_2019} is the entropically regularized optimal transport cost $\OTe$, with regularization parameter $\eps>0$, from which the two diagonal terms are subtracted,
\[
\Se(\mu,\nu)=\OTe(\mu,\nu)-\tfrac12\,\OTe(\mu,\mu)-\tfrac12\,\OTe(\nu,\nu).
\]
Under conditions on the cost (satisfied in the quadratic case), it is nonnegative, vanishes only on the diagonal and metrizes the weak-star convergence, and it interpolates between the optimal transport cost as $\eps\to0$ (the squared Wasserstein distance for the quadratic cost) and the squared maximum mean discrepancy of the kernel $-c/2$ as $\eps\to\infty$ \citep[Theorem~1]{feydy_sejourné_vialard_amari_trouvé_peyré_interpolating_2019}. \citet{lavenant_luckhardt_mordant_et_al_riemannian_2025} obtain its second-order term on the probability measures of a compact space, a metric tensor expressed through the Markov chain whose transition density is the density of the optimal coupling of a measure with itself (\autoref{subsec:gradients}, equation~\eqref{eq:Theta}), together with an embedding of the measures into the unit sphere of a reproducing kernel Hilbert space and the induced distance with its minimizing paths. In this paper we continue Eguchi's construction to the third order, on a Hilbert manifold of measures on which the metric tensor of Lavenant et al.\ is a weak Riemannian metric.

The metric and the distance of Lavenant et al.\ are defined along curves of measures, whereas connections require a differentiable structure, and the probability measures with the weak-star topology do not have one: no Banach manifold structure induces this topology, and the image of the sphere embedding is compact in the unit sphere of an infinite-dimensional Hilbert space (Remark~\ref{rem:nogo}). We obtain one in the way \citet{EbinMarsden1970} obtained Hilbert manifolds of diffeomorphisms in Sobolev classes: we work with the measures $\mu$ for which $\mu$ \emph{tilted by} $e^{\fmu/\eps}$, with $\fmu$ the Schr\"odinger potential of the self-transport problem, has a positive density $h$ in a Sobolev space $H^\sigma=H^\sigma(X)$ of the closed Riemannian manifold $X$, of order $\sigma$ above half the dimension so that its elements are continuous, restricting the regularity. In terms of $h$ the tilt and the operators of the chain are expressed through one fixed integral operator of the kernel (\autoref{subsec:operators} and \eqref{eq:faktorisierung}), whereas a density of $\mu$ determines the tilt only through a fixed-point equation. For kernels that raise the Sobolev order by a fixed amount, such as the Mat\'ern kernels on the torus, these measures form a real-analytic Hilbert manifold, the \emph{smooth stratum}, parametrized by a submanifold of codimension one of $H^\sigma$ on which the tilt and the operators of the chain depend analytically on the coordinates (Proposition~\ref{prop:stratum} and Corollary~\ref{cor:analytic}). The submanifold is defined by a single quadratic condition on $h$, which is at once the unit norm of the sphere embedding and the mass one of $\mu$. These are exactly the probability measures with a positive density in $H^\sigma$, so $h$ is a global analytic parametrization of the manifold of positive Sobolev densities, on which the density itself is a global chart, and over $\sigma$ the strata form a scale whose limit, as $C^\infty$ is the intersection of the Sobolev spaces, is the Fr\'echet manifold of smooth positive densities on which \citet{Ay2024Otto} works (Theorem~\ref{thm:scale}). In these coordinates the Sinkhorn metric is a bilinear form built from the integral operator of the kernel, the tilt and the resolvent $(\id-\Kmu)^{-1}$ on centered functions, where $\Kmu$ is the transition operator of the chain of Lavenant et al., which is reversible with invariant measure $\mu$ (\autoref{subsec:sinkhorn}). The flat map of the metric is injective and compact on every stratum, so the metric is a weak Riemannian metric in the sense of \citet{BauerBruverisMichor2014} there, and by restriction on Ay's manifold, where every Riemannian metric is weak (\autoref{subsec:stratummetric}).

Beside the second variation of the divergence, the metric admits a second characterization, following \citet{Ay2024Otto}, through the gradients it assigns to the mean values $\mu\mapsto\int_X f\,d\mu$ of continuous functions $f$. For the Sinkhorn metric these gradients exist if and only if $f$ belongs to a dense proper subspace of the continuous functions, for infinite $X$, and the map from $f$ to the gradient, the \emph{Riesz map} of the metric, is (on centered functions) the operator $\frac{2}{\eps}(\Kmu^{-1}-\Kmu)$ followed by multiplication with $\mu$ (\autoref{subsec:gradients}). Its cotangent form is the Dirichlet form of the chain reweighted by $(1+\lambda)/\lambda$ on the eigenvalues $\lambda$ of $\Kmu$, so the Sinkhorn metric is of the same kind as the Fisher-Rao and the Otto metric: the cotangent forms of all three are built from a process reversible for $\mu$, independent resampling from $\mu$ for Fisher-Rao, Brownian motion with drift $\nabla\log\rho$ for Otto, where $\rho$ is the density of $\mu$, and the Sinkhorn chain. On the stratum the two metrics are comparable in one direction only: the Sinkhorn metric is bounded by a multiple of the Fisher-Rao metric, with a constant that depends on the cost and on $\eps$ only, whereas the Fisher-Rao metric is not bounded by any multiple of the Sinkhorn metric (\autoref{subsec:stratummetric}). In the sphere coordinate of Lavenant et al.\ the metric is a multiple of the round metric of the unit sphere of the reproducing kernel Hilbert space plus a correction; we show that the correction is strictly positive on measures of full support, and we obtain the comparison of Lavenant et al.\ between the metric and the round metric with an improved constant, which comes from a Doeblin bound for the chain (\autoref{subsec:positivity}).

At third order the Sinkhorn divergence is self-dual: the two connection forms of Eguchi's construction coincide with each other and with the first-kind Levi-Civita form of the metric relative to the affine structure of the measures, and the cubic tensor vanishes (Theorem~\ref{thm:collapse}). On every smooth stratum the Levi-Civita form is \emph{represented}, i.e.\ for tangent vectors $u,v$ there is a tangent vector $\nabla_uv$ with $g(\nabla_uv,w)=\Gamma(u,v;w)$ for all $w$ (Remark~\ref{rem:sharp}), so the stratum has a Levi-Civita connection, real-analytic in the density chart, which on the Fr\'echet manifold of smooth positive densities is the connection of Ay and Schwachh\"ofer (Theorem~\ref{thm:levicivita}). The reason for the self-duality is simple: the transport problem is symmetric under the transposition $(x,y)\mapsto(y,x)$ of $X\times X$ and so is its value $\OTe$ in its two arguments; every symmetric divergence whose value is smooth along affine families is self-dual in this sense. A dual pair can instead be obtained from the metric, following \citet{Ay2024Otto} and \citet{AySchwachhoefer2025Torsion}, as the mixture connection and its dual with respect to the metric. The transposition symmetry does not make the metric a Hessian metric in the affine structure of the measures: on the torus, a computation in Fourier modes at the uniform measure shows that it is not one (Proposition~\ref{prop:nothessian}), so by \citet[Corollary~1]{AySchwachhoefer2025Torsion} the dual of the mixture connection is not torsion-free. In the second part of this work we study the geodesics and the curvature of the Sinkhorn metric and the torsion of this dual pair.

Like the non-parametric statistical manifolds of \citet{PistoneSempi1995}, \citet{Fukumizu2009ExponentialManifold} and \citet{newton_hilbertmanifold_jfa_2012,Newton2019Sobolev}, the smooth stratum is modelled on a function space, whereby it differs from them in the coordinate used. Pistone and Sempi, and Fukumizu with a reproducing kernel Hilbert space in place of their exponential Orlicz space, use the logarithm of the density as coordinate; Newton constructs Hilbert manifolds of measures and manifolds modelled on weighted Sobolev spaces, on which the Fisher-Rao metric is a weak Riemannian metric, as the Sinkhorn metric is on the stratum. The coordinate of the stratum is the density of the tilted measure, and its image under the integral operator of the kernel is the sphere coordinate of Lavenant et al., the counterpart of the square-root map through which \citet{KhesinLenellsMisiolekPreston2013} identify the Fisher-Rao geometry of smooth positive densities with that of the positive part of the sphere of smooth functions, with the $L^2$ metric; the difference is that the Sinkhorn metric in this coordinate is a multiple of the round metric plus a nonnegative correction, and to leading order as $\eps\to0$ the sphere coordinate becomes, after normalization, the square-root map (Remark~\ref{rem:epslimits}). In the framework of \citet{AyJostLeSchwachhoefer2017} the parametrization of the stratum is a parametrized measure model, so the Fisher metric and the Sinkhorn metric are defined on the same Hilbert manifold (Remark~\ref{rem:ajls}), the setting of the comparison in \autoref{subsec:stratummetric}.

On finite sets, \citet{AmariKarakidaOizumi2018} study the information geometry of entropic transport through the dually flat family of entropic plans, with a different divergence; for finitely supported measures whose atoms vary, \citet{MordantXu2026} prove that the Sinkhorn metric and all its derivatives converge as $\eps\to0$ to those of the Wasserstein metric of the atoms, and they list further applications of the metric.

\autoref{sec:prelim} and \autoref{subsec:lavmetric} recall the results from the literature that the paper uses: entropic transport and the Sinkhorn divergence, the operators of the self-transport problem, the sphere embedding and the metric of Lavenant et al. The results of this paper are in \autoref{sec:stratum}, with the proofs in Appendix~\ref{app:proofs-hilbert}, in Sections~\ref{subsec:gradients} to~\ref{subsec:stratummetric}, and in \autoref{sec:tensors} with Appendix~\ref{app:thirdvar}; \autoref{sec:discussion} discusses them together with open directions. The notation follows \citet{lavenant_luckhardt_mordant_et_al_riemannian_2025} wherever possible, and the regularization parameter $\eps>0$ is fixed throughout.

%% file: sec2.tex
\section{Entropic transport, the Sinkhorn divergence and the sphere embedding}\label{sec:prelim}

This section collects results from the lecture notes of \citet{Nutz2022EOTNotes} and the papers of \citet{feydy_sejourné_vialard_amari_trouvé_peyré_interpolating_2019} and \citet{lavenant_luckhardt_mordant_et_al_riemannian_2025} that the paper builds upon: the entropic transport problem with its dual, the Sinkhorn divergence, the reproducing kernel Hilbert space of the entropic kernel, the operators of the self-transport problem and the sphere embedding. The results are due to these authors and are stated with references to their proofs; the one addition is the constant $q$ of Proposition~\ref{prop:operators}, obtained from a Doeblin bound (\citet{Vialard2019EntropicNotes} obtains a constant of the same form for the Sinkhorn map).

\subsection{The entropic transport problem and the entropic kernel}\label{subsec:primal}

Let $X$ and $Y$ be compact metric spaces, let $\mu\in\mathcal P(X)$ and $\nu\in\mathcal P(Y)$ be probability measures on them, and let $c\in C(X\times Y)$ be a nonnegative continuous cost function. For a regularization parameter $\eps>0$, fixed throughout the paper, the entropically regularized optimal transport problem is
\begin{equation}\label{eq:eot-primal}
\OTe(\mu,\nu)
=
\inf_{\pi\in\Pi(\mu,\nu)}
\left\{
\int_{X\times Y} c(x,y)\,\pi(dx,dy)
+\eps\,\KL(\pi\mid\mu\otimes\nu)
\right\},
\end{equation}
where $\Pi(\mu,\nu)$ denotes the set of couplings of the marginals, that is, probability measures on $X\times Y$ whose first and second marginals are $\mu$ and $\nu$ respectively, and where the relative entropy of a probability measure $\pi$ with respect to a nonnegative finite measure $\rho$ is
\[
\KL(\pi\mid\rho)=\int\log\left(\frac{d\pi}{d\rho}\right)d\pi
\quad\text{if }\pi\ll\rho,
\qquad
\KL(\pi\mid\rho)=+\infty
\quad\text{otherwise.}
\]
The functional $\pi\mapsto\int c\,d\pi+\eps\,\KL(\pi\mid\mu\otimes\nu)$ is weak-star lower semicontinuous and strictly convex on the convex, weak-star compact set $\Pi(\mu,\nu)$, so the infimum in \eqref{eq:eot-primal} is attained, at a unique coupling $\pi_\eps$.

Write $k_c(x,y)=e^{-c(x,y)/\eps}$ for the \emph{entropic kernel}, a continuous function on $X\times Y$ with values in $(0,1]$.

\subsection{Duality and the Schr\"odinger potentials}\label{subsec:duality}

The problem \eqref{eq:eot-primal} admits an unconstrained concave dual. For $\mu\in\mathcal P(X)$ and $\nu\in\mathcal P(Y)$,
\begin{equation}\label{eq:eot-dual}
\begin{aligned}
\OTe(\mu,\nu)
=
\sup_{(f,g)\in C(X)\times C(Y)}
\Big\{
&\int_X f\,d\mu+\int_Y g\,d\nu\\
&-\eps\Big(\int_{X\times Y}e^{(f(x)+g(y)-c(x,y))/\eps}\,d(\mu\otimes\nu)-1\Big)
\Big\},
\end{aligned}
\end{equation}
without duality gap \citep[Theorem~4.7]{Nutz2022EOTNotes}; see also \citep{Leonard2014}.

Maximizers of \eqref{eq:eot-dual} exist and are unique $\mu$-a.e. and $\nu$-a.e. up to the shift $(f,g)\mapsto(f+\lambda,\,g-\lambda)$, $\lambda\in\mathbb R$ \citep[Theorem~4.2]{Nutz2022EOTNotes}; the \emph{Schr\"odinger potentials} $(f_{\mu,\nu},g_{\mu,\nu})$ are the continuous versions defined on all of $X$ and $Y$ by the Schr\"odinger system \eqref{eq:schroedingersystem} below \citep[Section~4.1]{Nutz2022EOTNotes}. At the optimum the exponential integral in \eqref{eq:eot-dual} equals one, so the value reduces to
\begin{equation}\label{eq:eot-value}
\OTe(\mu,\nu)=\int_X f_{\mu,\nu}\,d\mu+\int_Y g_{\mu,\nu}\,d\nu,
\end{equation}
and the unique optimal coupling of \eqref{eq:eot-primal} is the product measure \emph{tilted}, that is, multiplied, by the exponential factor of the potentials and the cost,
\begin{equation}\label{eq:eot-plan}
\pi_\eps=e^{(f_{\mu,\nu}(x)+g_{\mu,\nu}(y)-c(x,y))/\eps}\,(\mu\otimes\nu).
\end{equation}
The first-order conditions of \eqref{eq:eot-dual} are written with the map
\begin{equation}\label{eq:Teps}
T_\eps: C(X)\times\mathcal P(X)\to C(Y),
\qquad
T_\eps(f,\mu)(y)=-\eps\log\int_X e^{(f(x)-c(x,y))/\eps}\,d\mu(x),
\end{equation}
and likewise $T_\eps(g,\nu)\in C(X)$ for $g\in C(Y)$ and $\nu\in\mathcal P(Y)$, with the roles of $X$ and $Y$ exchanged. As $\eps\to0$, $T_\eps(f,\mu)(y)$ converges to the $c$-transform $\inf_{x\in\operatorname{supp}\mu}\left(c(x,y)-f(x)\right)$, and we call $T_\eps$ the \emph{entropic $c$-transform}. The objective $J(f,g)$ of \eqref{eq:eot-dual} satisfies $J(f+\lambda,g-\lambda)=J(f,g)$ for $\lambda\in\mathbb R$, and correspondingly $T_\eps(f+\lambda,\mu)=T_\eps(f,\mu)-\lambda$.

In terms of \eqref{eq:Teps}, the optimality conditions of \eqref{eq:eot-dual} are the \emph{Schr\"odinger system}
\begin{equation}\label{eq:schroedingersystem}
f_{\mu,\nu}=T_\eps(g_{\mu,\nu},\nu),
\qquad
g_{\mu,\nu}=T_\eps(f_{\mu,\nu},\mu).
\end{equation}
Since $T_\eps(f,\mu)$ depends on $f$ only through its $\mu$-a.e. class and is continuous, the system defines the potentials everywhere from the a.e. classes of the maximizers, and a continuous pair that solves it is unique up to the shift: for two such pairs, $f'=f+\lambda$ $\mu$-a.e. by the a.e. uniqueness, so $g'=T_\eps(f',\mu)=g-\lambda$ and $f'=T_\eps(g',\nu)=f+\lambda$ everywhere. Exponentiated, the two equations of \eqref{eq:schroedingersystem} state that the marginals of the tilt \eqref{eq:eot-plan} are $\mu$ and $\nu$.

\subsection{The Sinkhorn divergence and the self-transport problem}\label{subsec:sinkhorn}

From this point on the two marginal spaces coincide and the cost is symmetric, so that a measure can be compared with itself.

\begin{assumption}\label{assump:standing}
$X$ is a compact metric space; the cost $c\in C(X\times X)$ is symmetric, nonnegative, Lipschitz continuous,\footnote{The divergence properties of \autoref{subsec:sinkhorn} \citep[Theorem~1]{feydy_sejourné_vialard_amari_trouvé_peyré_interpolating_2019} use the Lipschitz continuity of $c$; the metric construction of \autoref{sec:metric} requires only continuity \citep{lavenant_luckhardt_mordant_et_al_riemannian_2025}.} and vanishes on the diagonal, $c(x,x)=0$ for all $x\in X$; and the entropic kernel $k_c=e^{-c/\eps}$ is strictly positive definite and universal on $X$: for every nonzero finite signed measure $\xi$ on $X$,
\[
\int_X\int_X k_c(x,y)\,d\xi(x)\,d\xi(y)>0,
\]
and the linear span of the functions $\left\{k_c(x,\cdot):x\in X\right\}$ is dense in $\left(C(X),\|\cdot\|_\infty\right)$.
\end{assumption}

The quadratic cost $c(x,y)=\|x-y\|^2$ on a compact subset of $\mathbb R^d$ satisfies all these conditions: its entropic kernel is the Gaussian kernel $e^{-\|x-y\|^2/\eps}$, which is strictly positive definite and universal \citep{feydy_sejourné_vialard_amari_trouvé_peyré_interpolating_2019}. On a compact space the two conditions on $k_c$ are equivalent \citep[Proposition~4 and Remark~3(a)]{SriperumbudurFukumizuLanckriet2011}.

The regularized cost $\OTe$ is not a divergence: $\OTe(\mu,\mu)>0$ in general, and $\nu\mapsto\OTe(\mu,\nu)$ is not minimized at $\nu=\mu$. \citet{Genevay2018SinkhornDivergences} remove the bias by subtracting the self-comparison terms, and call the result the \emph{Sinkhorn divergence},
\begin{equation}\label{eq:sinkhorndivergenz}
\Se(\mu,\nu)=\OTe(\mu,\nu)-\frac12\,\OTe(\mu,\mu)-\frac12\,\OTe(\nu,\nu);
\end{equation}
\citet{feydy_sejourné_vialard_amari_trouvé_peyré_interpolating_2019} analyze it: by construction $\Se(\mu,\mu)=0$, and by their Theorem~1, under Assumption~\ref{assump:standing} the Sinkhorn divergence $\Se:\PX\times\PX\to[0,\infty)$ is symmetric and nonnegative, with $\Se(\mu,\nu)=0$ if and only if $\mu=\nu$, convex in each of its arguments separately, and it metrizes the weak-star convergence of probability measures.

For the pair $(\mu,\mu)$ the Schr\"odinger system \eqref{eq:schroedingersystem} is invariant under the exchange of the two potentials, so by the uniqueness up to the shift exactly one of its continuous solutions has equal components, $f_{\mu,\mu}=g_{\mu,\mu}=:\fmu$, the \emph{self-potential} of $\mu$. It is the fixed point of the entropic $c$-transform,
\begin{equation}\label{eq:fixpunkt}
\fmu=T_\eps(\fmu,\mu),
\end{equation}
or equivalently, after exponentiating,
\begin{equation}\label{eq:selfnorm}
\int_X e^{(\fmu(x)+\fmu(y)-c(x,y))/\eps}\,d\mu(x)=1
\qquad\text{for all }y\in X.
\end{equation}
From the self-potential we form the \emph{tilt} $\amu:=e^{\fmu/\eps}$, a positive continuous function, and the \emph{self-kernel}
\begin{equation}\label{eq:selbstkern}
k_\mu(x,y):=e^{(\fmu(x)+\fmu(y)-c(x,y))/\eps}=\amu(x)\,k_c(x,y)\,\amu(y),
\end{equation}
a symmetric positive kernel which, by \eqref{eq:selfnorm}, satisfies $\int_X k_\mu(x,y)\,d\mu(x)=1$ for every $y$. So the self-kernel is a reversible Markov transition density with invariant measure $\mu$ \citep{lavenant_luckhardt_mordant_et_al_riemannian_2025}, and the optimal self-coupling $\pi_\mu=k_\mu\cdot(\mu\otimes\mu)$ of \eqref{eq:eot-plan} is the law of one step of the corresponding chain.

\subsection{The reproducing kernel Hilbert space of the entropic kernel}\label{subsec:rkhs}

Under Assumption~\ref{assump:standing}, $k_c:X\times X\to(0,1]$ is a continuous, symmetric and positive definite kernel on the compact space $X$, so by the Moore-Aronszajn theorem \citep[Chapter~1, Theorem~3]{BerlinetThomasAgnan2004} it is the reproducing kernel of a unique Hilbert space $\Hc$ of functions on $X$, which contains the sections $k_c(x,\cdot)$, $x\in X$, and in which
\[
\left\langle k_c(x,\cdot),k_c(y,\cdot)\right\rangle_{\Hc}=k_c(x,y),
\qquad
\varphi(x)=\left\langle\varphi,k_c(x,\cdot)\right\rangle_{\Hc}
\quad\text{for }\varphi\in\Hc.
\]
Since $k_c(x,x)=1$, the Cauchy-Schwarz inequality yields $|\varphi(x)|\le\|\varphi\|_{\Hc}$ for $\varphi\in\Hc$ and $x\in X$, so $\Hc$ consists of continuous functions and the inclusion $\Hc\to C(X)$ has norm at most one.

A finite signed measure $\xi\in\MX$ is mapped into $\Hc$ by integrating the sections:
\begin{equation}\label{eq:kme}
\Hc[\cdot]:\MX\to\Hc,
\qquad
\Hc[\xi]=\int_X k_c(x,\cdot)\,d\xi(x),
\end{equation}
pointwise $\Hc[\xi](y)=\int_X k_c(x,y)\,d\xi(x)$; this is the kernel mean embedding, and the symbol $\Hc$ denotes both the space and the map. The reproducing property yields, for every $\varphi\in\Hc$,
\begin{equation}\label{eq:kme-duality}
\left\langle\Hc[\xi],\varphi\right\rangle_{\Hc}
=\int_X\varphi\,d\xi,
\end{equation}
and for $\varphi=\Hc[\xi]$
\begin{equation}\label{eq:mmd-norm}
\left\|\Hc[\xi]\right\|_{\Hc}^{2}
=\int_X\int_X k_c(x,y)\,d\xi(x)\,d\xi(y),
\end{equation}
which for $\xi=\mu-\nu$ is the squared maximum mean discrepancy between $\mu$ and $\nu$ for the kernel $k_c$ \citep{RamdasGarciaTrillosCuturi2017,feydy_sejourné_vialard_amari_trouvé_peyré_interpolating_2019}. By the strict positive definiteness in Assumption~\ref{assump:standing} the right side of \eqref{eq:mmd-norm} is positive for $\xi\ne0$, so the kernel mean embedding is injective.

For $\mu\in\PX$ the self-kernel $k_\mu:X\times X\to(0,\infty)$ of \eqref{eq:selbstkern} is continuous, symmetric and positive definite, since $k_\mu(x,y)=\amu(x)\,k_c(x,y)\,\amu(y)$ with $\amu$ positive and continuous; it has the reproducing kernel Hilbert space $\Hmu$, and \eqref{eq:kme} with $k_\mu$ in place of $k_c$ defines the embedding $\Hmu[\cdot]:\MX\to\Hmu$. The factorization in \eqref{eq:selbstkern} identifies $\Hmu$ with $\Hc$:

\begin{lemma}[{\citep[Proposition~4.7]{lavenant_luckhardt_mordant_et_al_riemannian_2025}}]\label{lem:tilt-isometry}
For every $\mu\in\PX$ the map $\iota_\mu:\Hc\to\Hmu$, $\varphi\mapsto\amu\varphi$, is an isometric isomorphism.
\end{lemma}

By \eqref{eq:selfnorm}, $\Hmu[\mu]$ is the constant function $1$, and by \eqref{eq:mmd-norm} for the kernel $k_\mu$,
\[
\|1\|_{\Hmu}^{2}=\int_X\int_X k_\mu\,d\mu\,d\mu=1;
\]
Lemma~\ref{lem:tilt-isometry} then implies that $\amu^{-1}=e^{-\fmu/\eps}$ is a unit vector of $\Hc$, and the bound $|\varphi(x)|\le\|\varphi\|_{\Hc}$ yields $\fmu\ge0$ on $X$ for every $\mu\in\PX$ \citep[Remark~4.9]{lavenant_luckhardt_mordant_et_al_riemannian_2025}.

\subsection{The transition operator of the self-transport problem}\label{subsec:operators}

Fix $\mu\in\PX$. Composing the embedding $\Hmu[\cdot]$ of \autoref{subsec:rkhs} with the inclusion $\Hmu\to C(X)$, and restricting it to the measures $\varphi\mu$ with $\varphi\in C(X)$, yields the integral operators with kernel $k_\mu$ of \citet[Definition~3.3]{lavenant_luckhardt_mordant_et_al_riemannian_2025}, $\Hmu[\cdot]:\MX\to C(X)$ and $\Kmu:C(X)\to C(X)$,
\begin{equation}\label{eq:Kmu}
\Hmu[\nu](y)=\int_X k_\mu(x,y)\,d\nu(x),
\qquad
(\Kmu\varphi)(y)=\Hmu[\varphi\mu](y)=\int_X\varphi(x)\,k_\mu(x,y)\,d\mu(x).
\end{equation}
Both operators are compact as operators into $C(X)$ \citep[Proposition~3.6, proved in their Lemma~A.3]{lavenant_luckhardt_mordant_et_al_riemannian_2025}. By \eqref{eq:selfnorm}, $\Kmu 1=1$. So $\Kmu$ is the Markov operator of the reversible chain of \autoref{subsec:sinkhorn}, with transition density $k_\mu$ and invariant measure $\mu$.

% Schranke de Doeblin: k_mu >= k_c >= 1-q porque f_mu >= 0, mejor que la constante de Lavenant et al.
\begin{proposition}[{\citep[Propositions~3.8 and~4.2 and Theorems~3.9 and~4.3]{lavenant_luckhardt_mordant_et_al_riemannian_2025}}]\label{prop:operators}
Let Assumption~\ref{assump:standing} hold, let $\mu\in\PX$, and set $q:=1-e^{-\|c\|_\infty/\eps}\in(0,1)$. Then: (i) for all $\varphi,\psi\in\Hmu$, $\left\langle\Kmu\varphi,\psi\right\rangle_{\Hmu}=\left\langle\varphi,\psi\right\rangle_{L^2(\mu)}$; (ii) $\Kmu$ is self-adjoint, nonnegative and compact on $L^2(\mu)$ and on $\Hmu$, the eigenvalue $1$ is simple, with the constant functions as eigenspace, and the rest of the spectrum lies in $[0,q]$; (iii) $\id+\Kmu$ is boundedly invertible on $C(X)$, and on the mean-zero subspaces of $L^2(\mu)$ and $\Hmu$, for every integer $m\ge1$, as quadratic forms,
\[
\id\le\left(\id-\Kmu^{m}\right)^{-1}\le\left(1-q^{m}\right)^{-1}\id.
\]
\end{proposition}

Part (i) is \eqref{eq:kme-duality} for the kernel $k_\mu$. The constant $q$ comes from the pointwise bound $k_\mu\ge k_c\ge1-q$: since $\fmu\ge0$, $k_\mu(x,y)=\amu(x)\,k_c(x,y)\,\amu(y)\ge k_c(x,y)\ge e^{-\|c\|_\infty/\eps}$ for $x,y\in X$, so by \eqref{eq:selfnorm} the functions $p_y:=\left(k_\mu(\cdot,y)-(1-q)\right)/q$, $y\in X$, are probability densities with respect to $\mu$, and $(\Kmu\varphi)(y)-(\Kmu\varphi)(y')=q\int_X\varphi\,(p_y-p_{y'})\,d\mu$, which is at most $q\,(\sup\varphi-\inf\varphi)$ in absolute value. This Doeblin kernel minorization argument is the contraction of the oscillation in \citep[Proposition~3.8]{lavenant_luckhardt_mordant_et_al_riemannian_2025} with the constant $q$ above, and the proofs of \citep[Proposition~4.2 and Theorems~3.9 and~4.3]{lavenant_luckhardt_mordant_et_al_riemannian_2025} then yield (ii) and (iii). A contraction constant of the same form, $1-e^{-L\operatorname{diam}(\operatorname{supp}\nu)/\eps}$ for an $L$-Lipschitz cost and a marginal $\nu$, is obtained for the Sinkhorn map by \citet[Proposition~19]{Vialard2019EntropicNotes} from an $L^1$ bound between Gibbs measures \citep[Lemma~20]{Vialard2019EntropicNotes}. The constants of \citet{lavenant_luckhardt_mordant_et_al_riemannian_2025}, from bounds on ratios of values of the kernel, are $q=1-e^{-4\|c\|_\infty/\eps}$ \citep[Proposition~3.8]{lavenant_luckhardt_mordant_et_al_riemannian_2025} and $1-e^{-5\|c\|_\infty/(2\eps)}$ \citep[Remark~4.9]{lavenant_luckhardt_mordant_et_al_riemannian_2025}.

Since $1$ is an eigenvalue of $\Kmu$ with the constants as eigenspace, $\id-\Kmu^{m}$ is invertible only on a complement of the constants. We take the orthogonal complement in $L^2(\mu)$ and the projection onto it,
\begin{equation}\label{eq:meanzero}
L^2_0(\mu):=\left\{h\in L^2(\mu):\int_X h\,d\mu=0\right\},
\qquad
\Pi_0:L^2(\mu)\to L^2_0(\mu),
\quad
\Pi_0h:=h-\int_X h\,d\mu.
\end{equation}
From now on $\left(\id-\Kmu^{m}\right)^{-1}$ denotes the operator $\left(\id-\Kmu^{m}\right)^{-1}\Pi_0:L^2(\mu)\to L^2_0(\mu)$, which maps constants to zero; we call this the \emph{resolvent convention}. By Proposition~\ref{prop:operators}(iii), its restriction to $L^2_0(\mu)$ is bounded, self-adjoint and bounded below by the identity. It maps $C(X)$ into $C(X)$: $\left(\id-\Kmu^{m}\right)g=\Pi_0h$ yields $g=\Pi_0h+\Kmu^{m}g$, and $\Kmu$ maps $L^2(\mu)$ into $C(X)$. Identities between continuous functions, such as $\left(\id-\Kmu^{m}\right)g=\Pi_0h$ or $\Pi_0f=\Kmu\psi$, are identities in $C(X)$, on all of $X$ (pointwise).\footnote{When $\mu$ has full support they are equivalent to the same identities in $L^2(\mu)$, but in general they are stronger, since a continuous function that vanishes $\mu$-a.e. need not vanish off the support of $\mu$.}

\subsection{\texorpdfstring{The sphere embedding into $\Hc$}{The sphere embedding into H\_c}}\label{subsec:sphere}

The tilt $\amu=e^{\fmu/\eps}$ of \autoref{subsec:sinkhorn} and the unit vector $\amu^{-1}=e^{-\fmu/\eps}$ of $\Hc$ define two maps on $\PX$,
\begin{equation}\label{eq:Bmap}
A:\PX\to\MpX,\quad A(\mu)=\amu\mu,
\qquad
B:\PX\to\Hc,\quad B(\mu)=e^{-\fmu/\eps},
\end{equation}
the \emph{tilted measure} and the \emph{sphere embedding}. Multiplying \eqref{eq:selfnorm} by $e^{-\fmu(y)/\eps}$ yields $B(\mu)=\Hc[A(\mu)]$, the kernel mean embedding of the tilted measure, and $\mu=B(\mu)\,A(\mu)$ as measures, $B(\mu)$ being the function $\amu^{-1}$. For $w\in\Hc$ the pairing in \eqref{eq:kme-duality} yields
\begin{equation}\label{eq:tangency}
\left\langle B(\mu),w\right\rangle_{\Hc}=\int_X w\,dA(\mu)=\int_X\amu\,w\,d\mu ,
\end{equation}
so the tangent space $B(\mu)^{\perp}=\{w\in\Hc:\langle B(\mu),w\rangle_{\Hc}=0\}$ of the unit sphere of $\Hc$ at $B(\mu)$ consists of the $w$ for which $\amu w$ has $\mu$-mean zero; the velocity of a curve in the image of $B$ that is differentiable in $\Hc$ lies in it, since $\|B(\mu)\|_{\Hc}=1$ (\autoref{subsec:rkhs}). \citet{lavenant_luckhardt_mordant_et_al_riemannian_2025} use $B$ as a coordinate on all of $\PX$:

\begin{theorem}[{\citep[Theorem~4.8]{lavenant_luckhardt_mordant_et_al_riemannian_2025}}]\label{thm:sphere}
Let Assumption~\ref{assump:standing} hold. Then $B$ is injective, and it is a homeomorphism from $\PX$ with the weak-star topology onto its image $\mathcal B:=B(\PX)$ with the norm topology of $\Hc$; in particular $\mathcal B$ is norm compact.
\end{theorem}

Injectivity comes from the identity $\mu=B(\mu)\,A(\mu)$ together with the injectivity of the kernel mean embedding, which recovers $A(\mu)$ from $B(\mu)$; Lavenant et al.\ identify $\mathcal B$ with the unit vectors of $\Hc[\MpX]$ and invert $B$ explicitly.

\begin{example}[Dirac masses]\label{ex:dirac}
For $x\in X$ and $f=c(x,\cdot)$, \eqref{eq:Teps} yields $T_\eps(f,\delta_x)=c(x,\cdot)-f(x)=f$, since $f(x)=c(x,x)=0$, so $(f,f)$ solves the Schr\"odinger system \eqref{eq:schroedingersystem} for $(\delta_x,\delta_x)$, and $f_{\delta_x}=c(x,\cdot)$, $A(\delta_x)=\delta_x$ and $B(\delta_x)=k_c(x,\cdot)=\Hc[\delta_x]$: the sphere embedding maps the Dirac masses to the sections of the kernel.
\end{example}

%% file: sec3.tex
\section{The smooth stratum}\label{sec:stratum}

This paper builds the information geometry of the Sinkhorn divergence, its metric and its connections, on a manifold of probability measures. That manifold is the \emph{smooth stratum} we describe in this section, a Hilbert manifold of measures with Sobolev-regular densities, on which the connections act as operators on vector fields and define geodesic equations. Both require a differentiable structure, which the weak-star topology of $\PX$ we have seen so far does not admit.

\begin{remark}[No Banach manifold structure for the weak-star topology]\label{rem:nogo}
For infinite $X$, no Banach manifold structure on $\PX$ induces the weak-star topology: $\PX$ is compact (hence locally compact) and convex (hence connected), so by Riesz's theorem, such a manifold is modelled on some $\mathbb R^n$, and by invariance of domain it cannot contain the simplex $\Delta^{n+1}$, whereas $\lambda\mapsto\sum_{i=0}^{n+1}\lambda_i\delta_{x_i}$ embeds $\Delta^{n+1}$ homeomorphically for distinct $x_0,\dots,x_{n+1}\in X$. By Theorem~\ref{thm:sphere} the same holds for $\mathcal B$, which is norm compact and therefore has empty interior in the unit sphere of the infinite-dimensional space $\Hc$.
\end{remark}

The smooth stratum consists of the measures with a positive density of Sobolev regularity $\sigma$, and its topology is strictly finer than the weak-star topology. For $\sigma'>\sigma$ the stratum of regularity $\sigma'$ is contained in that of regularity $\sigma$ (Theorem~\ref{thm:scale}), and the intersection of all strata is the Fr\'echet manifold of smooth positive densities on which \citet{Ay2024Otto} works.\footnote{The strata are nested and all of infinite dimension, unlike the strata of a stratified space, which are disjoint manifolds of different dimensions.} We construct the differentiable structure by the approach of \citet{EbinMarsden1970}, who obtained Hilbert manifolds of diffeomorphisms by restricting to Sobolev regularity \citep[Section~7]{BauerBruverisMichor2014}, and we take as coordinate the density of the \textit{tilted measure} $\amu\mu$: if it has a density $h$ with respect to a fixed reference measure, the tilt and the operators of \autoref{subsec:operators} are expressed through $h$ and one fixed linear operator (Proposition~\ref{prop:stratum}), whereas a density of $\mu$ determines the tilt only through the fixed-point equation \eqref{eq:fixpunkt}. On the resulting Hilbert manifold the metric of \autoref{sec:metric} is a weak Riemannian metric (\autoref{subsec:stratummetric}): the map from tangent to cotangent spaces is injective but not surjective \citep[Section~2]{BauerBruverisMichor2014}.

From now on, wherever the smooth stratum is used, $X$ is, in addition to Assumption~\ref{assump:standing}, assumed to be a connected closed Riemannian manifold of dimension $d$ with Riemannian volume $\mathrm{vol}$, and $\mu_0\in\PX$ has a smooth positive density $\rho_0=d\mu_0/d\mathrm{vol}$. We write $\langle\cdot,\cdot\rangle_0$ for the inner product of $L^2(\mu_0)$. The regularity of a measure is expressed on the Sobolev scale $H^r(X)$, $r\ge0$: a finite measure has \emph{Sobolev regularity} $r$ if it equals $h\mu_0$ with $h\in H^r(X)$. Since multiplication by a smooth positive function is an isomorphism of $H^r(X)$, this class does not depend on the choice of $\mu_0$. For $r>d/2$ the space $H^r(X)$ embeds continuously into $\CX$, so its positive functions form an open subset, which we use in Proposition~\ref{prop:stratum}. The coordinates $h$ of the stratum will lie in $H^\sigma(X)$ with $\sigma>d/2$. The kernel mean embedding of a measure $h\mu_0$ is $\Hc[h\mu_0]=Sh$ with the integral operator
\begin{equation}\label{eq:Sop}
S:L^2(\mu_0)\to\CX,
\qquad
(Sh)(y)=\int_X k_c(x,y)\,h(x)\,d\mu_0(x).
\end{equation}
On $L^2(\mu_0)$ it is bounded and self-adjoint, since $k_c$ is continuous and symmetric, and it is nonnegative, since by \eqref{eq:mmd-norm}
\[
\langle h,Sh\rangle_0=\int_X\int_X k_c(x,y)\,h(x)\,h(y)\,d\mu_0(x)\,d\mu_0(y)=\left\|\Hc[h\mu_0]\right\|_{\Hc}^{2}.
\]

On the smooth stratum, the coordinate of a measure is the density of its tilted measure. For a positive $h\in\CX$ let
\begin{equation}\label{eq:hchart}
\mu_h:=h\,(Sh)\,\mu_0,
\end{equation}
a measure with a positive continuous density and total mass $\mu_h(X)=\langle h,Sh\rangle_0$. The constraint $\langle h,Sh\rangle_0=1$ normalizes $\mu_h$ and also implies by \eqref{eq:mmd-norm} that $Sh=\Hc[h\mu_0]$, which Proposition~\ref{prop:stratum}(ii) identifies with $B(\mu_h)$, is of unit norm.

\begin{assumption}[smooth regime]\label{assump:stratum}
There are exponents $s>d/2$ and $\sigma>d/2$ such that (M1) $S:H^r(X)\to H^{r+2s}(X)$ is an isomorphism for every $r\in[0,\sigma]$.
\end{assumption}

By (M1), $S$ raises Sobolev regularity by exactly $2s$ for every order $r\in[0,\sigma]$. Since $S$ is the integral operator of $k_c$ with respect to $\mathrm{vol}$ composed with multiplication by $\rho_0$, an isomorphism of each $H^r(X)$, condition (M1) depends only on the kernel and not on the choice of $\mu_0$. For every $\sigma>d/2$, we also use the fact (M2) that $H^\sigma(X)$ is a Banach algebra under pointwise multiplication, and $g\mapsto1/g$ is real-analytic on $\left\{g\in H^\sigma(X):g>0\right\}$, since by the Moser estimates \citep[Chapter~13, Section~3]{Taylor2011PDE3} the positive functions are invertible elements of $H^\sigma(X)$, and inversion is real-analytic on the invertible elements of a Banach algebra.

The running example is the flat torus $X=\mathbb T^d=\mathbb R^d/\mathbb Z^d$ with the Mat\'ern kernel \citep{BorovitskiyTereninMostowskyDeisenroth2020}; on $\mathbb R^d$ the Mat\'ern kernels range from the exponential kernel $e^{-\kappa|x-y|}$, whose cost is $\eps\kappa|x-y|$, to the Gaussian kernel of the quadratic cost in the limit of infinite smoothness \citep[Section~4.2.1]{RasmussenWilliams2006}, and on the flat torus, where the Laplacian is diagonal in the Fourier basis, convolution with a Mat\'ern kernel is a Fourier multiplier, so that (M1) reduces to \textit{the decay of its Fourier coefficients}. For $\kappa>0$ and $s\ge(d+1)/2$, let $k_c(x,y)=G(x-y)/G(0)$, where $G$ is the Green's function of $(\kappa^2-\Delta)^s$ on $\mathbb T^d$, with Fourier coefficients $\hat G(n)=(\kappa^2+4\pi^2|n|^2)^{-s}$, $n\in\mathbb Z^d$. The function $G$ is positive, as a superposition of heat kernels. Since all $\hat G(n)$ are positive, $\int_X\int_X k_c\,d\nu\,d\nu=G(0)^{-1}\sum_n\hat G(n)|\hat\nu(n)|^2>0$ for every nonzero $\nu\in\MX$; the norm of $\Hc$ is $\|\varphi\|_{\Hc}^{2}=G(0)\sum_n(\kappa^2+4\pi^2|n|^2)^s|\hat\varphi(n)|^2$, so $\Hc=H^s(\mathbb T^d)$ with an equivalent norm \citep[Chapter~1, Remark~7.6]{LionsMagenes1972}, and it contains the trigonometric polynomials. Hence $k_c$ is strictly positive definite and universal. The cost $c(x,y)=\phi(x-y)$, with $\phi:=-\eps\log(G/G(0))$, is symmetric, and it is nonnegative and vanishes exactly on the diagonal, since $G(0)-G(z)=\sum_n\hat G(n)\left(1-\cos2\pi n\cdot z\right)$ is nonnegative and vanishes only at $z=0$ in $\mathbb T^d$. It is Lipschitz, since $G$ is Lipschitz and bounded below by a positive constant: for $s>(d+1)/2$ the Fourier series of $G$ can be differentiated termwise, and for $s=(d+1)/2$, $G$ is a positive multiple of $\sum_{m\in\mathbb Z^d}e^{-\kappa|z+m|}$ by the Poisson summation formula and \citep[Chapter~I, Theorem~1.14]{SteinWeiss1971}. Assumption~\ref{assump:standing} therefore holds. For (M1), $Sh=G\ast(h\rho_0)/G(0)$, and convolution with $G$ is an isomorphism from $H^r$ onto $H^{r+2s}$, so (M1) holds for every $\sigma$.

The quadratic cost lies outside the smooth regime: on the torus its Gaussian kernel becomes the normalized heat kernel, whose Fourier coefficients $e^{-4\pi^2t|n|^2}$ decay faster than any power, so $S$ maps $L^2(\mu_0)$ into $C^\infty(\mathbb T^d)$ but not onto any Sobolev space, and (M1) fails for every $s$.
%\footnote{Everything from Proposition~\ref{prop:stratum} to Theorem~\ref{thm:scale} use only that $S$ maps $H^r(X)$ into $H^{r+2s}(X)$, so the strata remain manifolds; the surjectivity in (M1) yields the Sobolev description of $\Hc$, \eqref{eq:sobolevskala}, and of the range of the flat map (Proposition~\ref{prop:range}), without which the Levi-Civita form of \autoref{sec:tensors} need not be represented.} 

\autoref{subsec:stratummetric} evaluates the metric on a velocity $v$ of the stratum through the norm $\|Sv\|_{\Hc}=\|S^{1/2}v\|_0$, which follows from \eqref{eq:mmd-norm} and is weaker than the norm of $H^\sigma(X)$, which is what makes the metric weak. The following lemma identifies $\Hc$ with the range of $S^{1/2}$.

\begin{lemma}\label{lem:douglas}
Let $\iota:\Hc\to L^2(\mu_0)$, $\varphi\mapsto\varphi$, be the inclusion and $\iota^{\ast}:L^2(\mu_0)\to\Hc$ its adjoint. Then $\iota^{\ast}\varphi=\Hc[\varphi\mu_0]=S\varphi$ for every $\varphi\in L^2(\mu_0)$, so $S=\iota\,\iota^{\ast}$ as an operator on $L^2(\mu_0)$. Moreover $S$ is injective, $\operatorname{ran}S^{1/2}=\Hc$, and $S^{1/2}:L^2(\mu_0)\to\Hc$ is an isometric isomorphism.
\end{lemma}

\newcommand{\PrfDouglas}{%
By the reproducing property, $\langle\iota^{\ast}\varphi,k_c(y,\cdot)\rangle_{\Hc}=\langle\varphi,k_c(y,\cdot)\rangle_0=(S\varphi)(y)$ for $y\in X$, so $\iota^{\ast}\varphi=S\varphi$, which is $\Hc[\varphi\mu_0]$ by \eqref{eq:kme}, and $S=\iota\,\iota^{\ast}$. By \eqref{eq:mmd-norm}, $\langle\varphi,S\varphi\rangle_0=\|\Hc[\varphi\mu_0]\|_{\Hc}^{2}$, which is positive for $\varphi\ne0$ by the strict positive definiteness in Assumption~\ref{assump:standing}, so $S$ is injective. For $\psi\in L^2(\mu_0)$, $\|S\psi\|_{\Hc}^{2}=\langle\psi,\iota\,\iota^{\ast}\psi\rangle_0=\|S^{1/2}\psi\|_0^{2}$, so the map $S^{1/2}\psi\mapsto S\psi=S^{1/2}(S^{1/2}\psi)$ is isometric from $\operatorname{ran}S^{1/2}$, which is dense in $L^2(\mu_0)$ since $S$ is injective, into $\Hc$. Its continuous extension to $L^2(\mu_0)$ is $S^{1/2}$, since $\Hc$ embeds continuously into $L^2(\mu_0)$. Its range is closed and contains $\operatorname{ran}\iota^{\ast}$, which is dense in $\Hc$ since $\iota$ is injective, so it is $\Hc$.
}
\iflong
\begin{proof}
\PrfDouglas
\end{proof}
\else
The proof is in Appendix~\ref{app:proofs-hilbert}.
\fi

In an orthonormal basis $(e_n)$ of $L^2(\mu_0)$ of eigenfunctions of the compact operator $S$, with eigenvalues $\lambda_n>0$, Lemma~\ref{lem:douglas} is the classical description $\Hc=\left\{\sum_na_ne_n:\sum_na_n^2/\lambda_n<\infty\right\}$ of Mercer's theorem. By (M1) for $r=0$, $S$ is an isomorphism of $L^2(\mu_0)$ onto $H^{2s}(X)$, so $\operatorname{dom}S^{-1}=\operatorname{ran}S=H^{2s}(X)$, the graph norm of $S^{-1}$ is equivalent to the norm of $H^{2s}(X)$, and
\begin{equation}\label{eq:sobolevskala}
H^{2s}(X)=\operatorname{ran}S\subset\Hc=\operatorname{ran}S^{1/2}\subset L^2(\mu_0).
\end{equation}
Since $S^{-1}$ is a positive self-adjoint operator on $L^2(\mu_0)$ with domain $H^{2s}(X)$, the interpolation space of order $1/2$ between $L^2(\mu_0)$ and $H^{2s}(X)$ is the domain of $S^{-1/2}$ \citep[Chapter~1, Definition~2.1 and Remark~2.3]{LionsMagenes1972}, and by the interpolation of Sobolev spaces on the closed manifold $X$ \citep[Chapter~1, Theorem~7.7 and Remark~7.6]{LionsMagenes1972} it is $H^s(X)$:
\[
\Hc=\operatorname{dom}S^{-1/2}=\left[H^{2s}(X),L^2(\mu_0)\right]_{1/2}=H^{s}(X)
\]
with equivalent norms. For $v\in L^2(\mu_0)$ and $\varphi\in\Hc$, $\langle v,\varphi\rangle_0=\langle S^{1/2}v,S^{-1/2}\varphi\rangle_0$, with equality in the Cauchy-Schwarz inequality at $\varphi=Sv$, so $\|S^{1/2}v\|_0$ is the norm of $v$ as a functional on $\Hc$; it is therefore equivalent to the norm of $H^{-s}(X)$, the dual of $H^s(X)$ with respect to $L^2(\mu_0)$ \citep[Chapter~1, Theorem~7.6]{LionsMagenes1972}, in which \autoref{subsec:stratummetric} measures velocities on the stratum.

\begin{definition}\label{def:stratum}
Let Assumption~\ref{assump:stratum} hold. The \emph{smooth stratum} is
\[
\mathcal P^\sigma(X):=\left\{\mu\in\PX:A(\mu)=h\mu_0\text{ for some positive }h\in H^\sigma(X)\right\},
\]
the set of probability measures whose tilted measure has a positive density of Sobolev regularity $\sigma$. Its parameter set is
\begin{equation}\label{eq:parametermenge}
U_\sigma:=\left\{h\in H^\sigma(X):h>0\text{ on }X,\quad\langle h,Sh\rangle_0=1\right\},
\end{equation}
with the map $U_\sigma\to\PX$, $h\mapsto\mu_h$, of \eqref{eq:hchart}.
\end{definition}

\begin{proposition}[Parametrization of the smooth stratum]\label{prop:stratum}
Let Assumptions~\ref{assump:standing} and~\ref{assump:stratum} hold. (i) $U_\sigma$ is a real-analytic submanifold of $H^\sigma(X)$ of codimension one, with tangent space $T_hU_\sigma=\left\{v\in H^\sigma(X):\langle v,Sh\rangle_0=0\right\}$. (ii) For $h\in U_\sigma$, the self-potential of $\mu_h$ is $f_{\mu_h}=-\eps\log(Sh)$; consequently $a_{\mu_h}=1/(Sh)$, $A(\mu_h)=h\mu_0$ and $B(\mu_h)=Sh$. Conversely, if $\mu\in\PX$ and $A(\mu)=h\mu_0$ with $h\in H^\sigma(X)$ positive, then $h\in U_\sigma$ and $\mu=\mu_h$. Hence the map $U_\sigma\to\mathcal P^\sigma(X)$, $h\mapsto\mu_h$, is a bijection. (iii) For $h\in U_\sigma$ and $v\in H^\sigma(X)$, $\langle B(\mu_h),Sv\rangle_{\Hc}=\int_X a_{\mu_h}\,Sv\,d\mu_h=\langle v,Sh\rangle_0$, so $T_hU_\sigma$ consists of the $v$ with $Sv$ in the tangent space $B(\mu_h)^{\perp}$ of the unit sphere; a curve in $U_\sigma$ differentiable in $H^\sigma(X)$ with velocity $v$ has a sphere image differentiable in $\Hc$ with velocity $Sv$.
\end{proposition}

\newcommand{\PrfStratum}{%
(i) The map $N:H^\sigma(X)\to\mathbb R$, $N(h)=\langle h,Sh\rangle_0$, is a continuous quadratic form, hence real-analytic, and since $S$ is self-adjoint on $L^2(\mu_0)$, its derivative is $DN(h)v=2\langle v,Sh\rangle_0$. At $h\in U_\sigma$, $DN(h)h=2N(h)=2$, so $DN(h)\ne0$, and the level set $\{N=1\}$ is regular on the set of positive functions, which is open in $H^\sigma(X)$. By the analytic implicit function theorem, $U_\sigma$ is a real-analytic submanifold of codimension one, with $T_hU_\sigma=\ker DN(h)$.

(ii) Set $f:=-\eps\log(Sh)$, which is continuous since $Sh$ is continuous and positive. Then $a:=e^{f/\eps}=1/(Sh)$, and for every $y\in X$, since $d\mu_h=h\,(Sh)\,d\mu_0$,
\[
\begin{aligned}
\int_X e^{(f(x)+f(y)-c(x,y))/\eps}\,d\mu_h(x)
&=a(y)\int_X k_c(x,y)\,a(x)\,h(x)\,(Sh)(x)\,d\mu_0(x)\\
&=a(y)\,(Sh)(y)=1.
\end{aligned}
\]
This is \eqref{eq:selfnorm}, so $(f,f)$ solves the Schr\"odinger system \eqref{eq:schroedingersystem} for $(\mu_h,\mu_h)$ with equal components, and $f_{\mu_h}=f$ by \autoref{subsec:sinkhorn}, as in Example~\ref{ex:dirac}. Hence $A(\mu_h)=a\,\mu_h=h\mu_0$ and $B(\mu_h)=e^{-f/\eps}=Sh$. Conversely, let $A(\mu)=h\mu_0$ with $h\in H^\sigma(X)$ positive. Then $B(\mu)=\Hc[A(\mu)]=Sh$ (\autoref{subsec:sphere}), so $\langle h,Sh\rangle_0=\|Sh\|_{\Hc}^{2}=\|B(\mu)\|_{\Hc}^{2}=1$ by \eqref{eq:mmd-norm} and \autoref{subsec:rkhs}, that is, $h\in U_\sigma$, and $\mu=B(\mu)\,A(\mu)=(Sh)\,h\,\mu_0=\mu_h$. The map $h\mapsto\mu_h$ is injective, since $\mu_h=\mu_{h'}$ implies $h\mu_0=A(\mu_h)=h'\mu_0$, hence $h=h'$, both being continuous and $\mu_0$ having full support.

(iii) By (ii), \eqref{eq:tangency} and $d\mu_h=h\,(Sh)\,d\mu_0$, $\langle B(\mu_h),Sv\rangle_{\Hc}=\int_X a_{\mu_h}\,Sv\,d\mu_h=\int_X(Sv)\,h\,d\mu_0=\langle v,Sh\rangle_0$, which vanishes exactly on $T_hU_\sigma$ by (i). By Lemma~\ref{lem:douglas}, $S=S^{1/2}S^{1/2}$ is bounded from $L^2(\mu_0)$ into $\Hc$, and $H^\sigma(X)$ embeds continuously into $L^2(\mu_0)$, so $\tau\mapsto Sh_\tau=B(\mu_{h_\tau})$ is differentiable in $\Hc$ with derivative $Sv$.
}
\iflong
\begin{proof}
\PrfStratum
\end{proof}
\else
The proof is in Appendix~\ref{app:proofs-hilbert}.
\fi

For $h\in U_\sigma$, with $a=a_{\mu_h}=1/(Sh)$ and $d\mu_h=h\,(Sh)\,d\mu_0$, the self-kernel \eqref{eq:selbstkern} satisfies $k_{\mu_h}(x,y)\,h(x)\,(Sh)(x)=a(y)\,k_c(x,y)\,h(x)$, so the transition operator \eqref{eq:Kmu} factorizes as
\begin{equation}\label{eq:faktorisierung}
K_{\mu_h}=M_a\,S\,M_h,
\end{equation}
where $M_g:\varphi\mapsto g\varphi$ denotes multiplication by $g$. We give $\mathcal P^\sigma(X)$ the topology that makes $h\mapsto\mu_h$ a homeomorphism from $U_\sigma$. It is finer than the weak-star topology, since $\mu_h=B^{-1}(Sh)$, where $h\mapsto Sh$ is continuous from $H^\sigma(X)$ into $\Hc$ and $B^{-1}$ is continuous by Theorem~\ref{thm:sphere}. It is strictly finer because $S:H^\sigma(X)\to\Hc$ is compact: for $h\in U_\sigma$, an orthonormal sequence $(w_n)$ in $T_hU_\sigma$ and small $t>0$, the functions $h_n:=(h+tw_n)/\langle h+tw_n,S(h+tw_n)\rangle_0^{1/2}$ lie in $U_\sigma$, and $\mu_{h_n}\to\mu_h$ weak-star while $\|h_n-h\|_{H^\sigma}\to t$.

By the following corollary, the metric of \autoref{sec:metric}, written in the coordinate $h$ in \autoref{subsec:stratummetric} through the tilt, the resolvents, $S$ and the inner product $\langle\varphi,h\,(Sh)\,\psi\rangle_0$ of $L^2(\mu_h)$, is real-analytic in $h$. The isomorphism property of (M1) also fixes the Sobolev order of the metric: it identifies $\Hc$ with $H^{s}(X)$, \eqref{eq:sobolevskala}, and makes the norm of the metric on the stratum equivalent to the norm of $H^{-s}(X)$ (Proposition~\ref{prop:stratumflat}(ii)).

\begin{corollary}[Analytic dependence on the coordinate]\label{cor:analytic}
Let Assumptions~\ref{assump:standing} and~\ref{assump:stratum} hold, let $m\ge1$ be an integer, and let $\mathcal L(L^2(\mu_0))$ be the Banach algebra of bounded operators on $L^2(\mu_0)$. Then the maps $U_\sigma\to H^\sigma(X)$, $h\mapsto a_{\mu_h}$, and $U_\sigma\to\mathcal L(L^2(\mu_0))$, $h\mapsto K_{\mu_h}$, $h\mapsto(\id+K_{\mu_h})^{-1}$, $h\mapsto(\id-K_{\mu_h}^{m})^{-1}$, are real-analytic, where the last operator is taken in the resolvent convention of \autoref{subsec:operators}.
\end{corollary}

\newcommand{\PrfAnalytic}{%
By (M1), $S$ maps $H^\sigma(X)$ boundedly into $H^{\sigma+2s}(X)\subset H^\sigma(X)$, and $Sh>0$ for $h>0$; so by (M2) and Proposition~\ref{prop:stratum}(ii), $h\mapsto a_{\mu_h}=1/(Sh)$ is real-analytic. The map $H^\sigma(X)\to\mathcal L(L^2(\mu_0))$, $g\mapsto M_g$, is linear and bounded, since $\|M_g\|\le\|g\|_\infty\le C\|g\|_{H^\sigma}$, so $h\mapsto K_{\mu_h}=M_{a_{\mu_h}}\,S\,M_h$ is real-analytic by \eqref{eq:faktorisierung}. Since $d\mu_h=h\,(Sh)\,d\mu_0$ with a density bounded above and below, $L^2(\mu_h)$ and $L^2(\mu_0)$ coincide with equivalent norms; by Proposition~\ref{prop:operators}(ii), $\id+K_{\mu_h}$ is therefore invertible in $\mathcal L(L^2(\mu_0))$, and inversion is real-analytic on the invertible elements of a Banach algebra. For the last map, let $P_h\varphi:=\langle\varphi,h\,Sh\rangle_0$, viewed as a constant function; this is the projection onto the constants along $L^2_0(\mu_h)$, and it depends analytically on $h$. The operator $\id-K_{\mu_h}^{m}+P_h$ acts as the identity on the constants and as $\id-K_{\mu_h}^{m}$ on $L^2_0(\mu_h)$, which $K_{\mu_h}$ preserves; so it is invertible by Proposition~\ref{prop:operators}(ii), and $(\id-K_{\mu_h}^{m})^{-1}=(\id-K_{\mu_h}^{m}+P_h)^{-1}-P_h$ is real-analytic.
}
\iflong
\begin{proof}
\PrfAnalytic
\end{proof}
\else
The proof is in Appendix~\ref{app:proofs-hilbert}.
\fi

A measure of the smooth stratum with coordinate $h$ has the density $\rho_0\,h\,(Sh)$ with respect to $\mathrm{vol}$, and the map
\[
\Phi_\sigma:U_\sigma\to H^\sigma(X),\qquad\Phi_\sigma(h)=\rho_0\,h\,(Sh),
\]
from coordinates to densities is real-analytic by (M1) and (M2). The following theorem identifies the smooth stratum with the manifold of positive probability densities of Sobolev regularity $\sigma$, on which $h$ is a global real-analytic parametrization and the density a global chart, and it brings the Sobolev scale of the strata together with the Fr\'echet setting of \citet{Ay2024Otto}: the strata are nested, a measure has the same coordinate $h$ on every stratum, the charts of the smaller strata are restrictions of those of the larger ones, and the intersection of the strata is Ay's manifold $\mathcal P^\infty_+(X)$ of smooth positive densities, as a set and as a Fr\'echet manifold, just as $C^\infty(X)$ is the intersection of the spaces $H^\sigma(X)$, with the Fr\'echet topology defined by their norms. Banach calculus is thus available on every stratum. For maps between open subsets of Fr\'echet spaces and their submanifolds we use the \emph{convenient calculus} of \citet{KrieglMichor1997}, the setting of \citet{Ay2024Otto}: a map is smooth if it maps smooth curves to smooth curves.

% C^infty como interseccion de los H^sigma, igual que en Ay; la coordenada h es la misma en todos los niveles
\begin{theorem}[The strata as a Sobolev scale with limit $\mathcal P^\infty_+(X)$]\label{thm:scale}
Let Assumptions~\ref{assump:standing} and~\ref{assump:stratum} hold.
(i) $\mathcal P^\sigma(X)=\left\{\rho\,\mathrm{vol}:\rho\in H^\sigma(X),\ \rho>0,\ \int_X\rho\,d\mathrm{vol}=1\right\}$, and $\Phi_\sigma$ is a real-analytic diffeomorphism from $U_\sigma$ onto the set of these densities, an open subset of the closed affine hyperplane $\left\{\rho\in H^\sigma(X):\int_X\rho\,d\mathrm{vol}=1\right\}$ of $H^\sigma(X)$.
(ii) If $\sigma'>\sigma$ and Assumption~\ref{assump:stratum} holds for $\sigma'$ in place of $\sigma$, then $\mathcal P^{\sigma'}(X)\subset\mathcal P^\sigma(X)$ and $U_{\sigma'}\subset U_\sigma$, a measure in $\mathcal P^{\sigma'}(X)$ has the same coordinate on both levels, $\Phi_{\sigma'}$ is the restriction of $\Phi_\sigma$, and the inclusion $U_{\sigma'}\to U_\sigma$ is real-analytic.
(iii) If Assumption~\ref{assump:stratum} holds for every $\sigma>d/2$, then
\[
\bigcap_{\sigma>d/2}\mathcal P^\sigma(X)=\mathcal P^\infty_+(X):=\left\{\rho\,\mathrm{vol}:\rho\in C^\infty(X),\ \rho>0,\ \int_X\rho\,d\mathrm{vol}=1\right\},
\]
the manifold of \citet{Ay2024Otto}; the set $U_\infty:=\bigcap_{\sigma>d/2}U_\sigma=\left\{h\in C^\infty(X):h>0,\ \langle h,Sh\rangle_0=1\right\}$ is a submanifold of codimension one of $C^\infty(X)$; and the common restriction $\Phi$ of the maps $\Phi_\sigma$ is a diffeomorphism from $U_\infty$ onto the densities of $\mathcal P^\infty_+(X)$, an open subset of a closed affine hyperplane of $C^\infty(X)$.
\end{theorem}

\newcommand{\PrfScale}{%
(i) Let $\mu\in\mathcal P^\sigma(X)$. By Proposition~\ref{prop:stratum}(ii), $\mu=\mu_h$ with $h\in U_\sigma$, so $d\mu/d\mathrm{vol}=\rho_0\,h\,(Sh)$. By (M1), $Sh\in H^{\sigma+2s}(X)\subset H^\sigma(X)$, and by (M2) the product $\rho_0\,h\,(Sh)$ lies in $H^\sigma(X)$; it is positive since $h>0$ and $k_c>0$, and its integral is $\mu_h(X)=1$. Conversely, let $\rho\in H^\sigma(X)$ be positive with $\int_X\rho\,d\mathrm{vol}=1$, and let $\mu=\rho\,\mathrm{vol}$. The self-potential $\fmu$ is continuous, as a fixed point \eqref{eq:fixpunkt} of the entropic $c$-transform, so the tilt $\amu=e^{\fmu/\eps}$ is continuous and positive, and $A(\mu)=h\mu_0$ with $h:=\amu\rho/\rho_0$, which is continuous and positive. The identity $B(\mu)=\Hc[A(\mu)]$ after \eqref{eq:Bmap} becomes $1/\amu=Sh$, so
\begin{equation}\label{eq:hdensity}
h=\frac{\rho}{\rho_0\,Sh}.
\end{equation}
Let $r_j=\min(\sigma,2js)$ for integers $j\ge0$. Then $h\in L^2(\mu_0)=H^{r_0}(X)$. If $h\in H^{r_j}(X)$, then $Sh\in H^{r_j+2s}(X)$ by (M1), since $r_j\le\sigma$; as $Sh>0$ and $r_j+2s>d/2$, (M2) yields $1/(\rho_0\,Sh)\in H^{r_j+2s}(X)$, and since $\rho\in H^\sigma(X)$, the right side of \eqref{eq:hdensity} lies in $H^{r_{j+1}}(X)$. After $\lceil\sigma/(2s)\rceil$ steps $r_j=\sigma$, so $h$ is a positive element of $H^\sigma(X)$, and the converse part of Proposition~\ref{prop:stratum}(ii) yields $h\in U_\sigma$ and $\mu=\mu_h\in\mathcal P^\sigma(X)$.

The map $\Phi_\sigma$ is the restriction to $U_\sigma$ of the quadratic map $g\mapsto\rho_0\,g\,(Sg)$ on $H^\sigma(X)$, which is continuous by (M1) and (M2), hence real-analytic. By the first part and the injectivity of $h\mapsto\mu_h$ in Proposition~\ref{prop:stratum}(ii), $\Phi_\sigma$ is a bijection of $U_\sigma$ onto the positive densities of mass one, which form an open subset of the affine hyperplane $\left\{\rho\in H^\sigma(X):\int_X\rho\,d\mathrm{vol}=1\right\}$, since $H^\sigma(X)$ embeds into $\CX$. On $H^\sigma(X)$ the derivative is
\[
D\Phi_\sigma(h)v=\rho_0\left(v\,(Sh)+h\,(Sv)\right)=\rho_0\,(Sh)\left(\id+M_{h/Sh}\,S\right)v.
\]
The operator $M_{h/Sh}\,S$ is compact on $H^\sigma(X)$: $S$ maps $H^\sigma(X)$ boundedly into $H^{\sigma+2s}(X)$ by (M1), the embedding of $H^{\sigma+2s}(X)$ into $H^\sigma(X)$ is compact on the closed manifold $X$, and multiplication by $h/Sh$ is bounded on $H^\sigma(X)$ by (M2). So $\id+M_{h/Sh}\,S$ is Fredholm of index zero. It is injective: if $v+(h/Sh)\,Sv=0$, multiplication by $Sh/h$ and the inner product with $v$ in $L^2(\mu_0)$ yield
\[
\int_X\frac{Sh}{h}\,v^{2}\,d\mu_0+\langle v,Sv\rangle_0=0,
\]
where both terms are nonnegative, the second since $S$ is nonnegative; as $Sh/h>0$ and $v$ is continuous, $v=0$. Multiplication by $\rho_0\,Sh$ is an isomorphism of $H^\sigma(X)$ by (M2), so $D\Phi_\sigma(h)$ is an isomorphism of $H^\sigma(X)$. By the self-adjointness of $S$, $\int_X D\Phi_\sigma(h)v\,d\mathrm{vol}=\langle v,Sh\rangle_0+\langle h,Sv\rangle_0=2\langle v,Sh\rangle_0$, so $D\Phi_\sigma(h)$ maps the closed hyperplane $T_hU_\sigma$ of $H^\sigma(X)$ into the closed hyperplane of densities of mass zero, and onto it, since an isomorphism maps closed hyperplanes to closed hyperplanes. By the analytic inverse function theorem, $\Phi_\sigma$ is a local real-analytic diffeomorphism at every point, and since it is bijective, it is a real-analytic diffeomorphism.

(ii) Assumption~\ref{assump:stratum} for $\sigma'$ contains Assumption~\ref{assump:stratum} for $\sigma$, since (M1) on $[0,\sigma']$ includes $[0,\sigma]$. The embedding of $H^{\sigma'}(X)$ into $H^\sigma(X)$ is continuous and the constraints defining $U_{\sigma'}$ and $U_\sigma$ are the same, so $U_{\sigma'}=U_\sigma\cap H^{\sigma'}(X)$. If $A(\mu)=h\mu_0$ with $h\in H^{\sigma'}(X)$, then $h\in H^\sigma(X)$; so $\mathcal P^{\sigma'}(X)\subset\mathcal P^\sigma(X)$, and the coordinate $h$, the density of $A(\mu)$, is the same on both levels. The maps $\Phi_{\sigma'}$ and $\Phi_\sigma$ are given by the same formula, and the inclusion $U_{\sigma'}\to U_\sigma$ is the restriction of a continuous linear map to an analytic submanifold, hence real-analytic.

(iii) By the Sobolev embeddings, $C^\infty(X)$ is the intersection of the spaces $H^\sigma(X)$, $\sigma>d/2$, with the Fr\'echet topology defined by their norms; so by (i) on every level and the identity of coordinates in (ii), a measure lies in every stratum if and only if its density with respect to $\mathrm{vol}$ is a positive smooth function of mass one, which proves $\bigcap_{\sigma>d/2}\mathcal P^\sigma(X)=\mathcal P^\infty_+(X)$, and $U_\infty=\left\{h\in C^\infty(X):h>0,\ N(h)=1\right\}$ with $N(g)=\langle g,Sg\rangle_0$ as in the proof of Proposition~\ref{prop:stratum}. For $h_0\in U_\infty$ and $\ell(g):=\langle g,Sh_0\rangle_0$, the map $g\mapsto\left(g/\ell(g)-h_0,\ N(g)\right)$ is a diffeomorphism from the positive functions in $C^\infty(X)$ onto $W\times(0,\infty)$, where $W:=\{v\in C^\infty(X):\ell(v)=0,\ h_0+v>0\}$ is open in a closed hyperplane, with inverse $(v,t)\mapsto\left(t/N(h_0+v)\right)^{1/2}(h_0+v)$, both maps being built from continuous linear maps, the continuous quadratic form $N$ and division by positive smooth real functions; it maps $U_\infty$ onto $W\times\{1\}$, so $U_\infty$ is a submanifold of codimension one of $C^\infty(X)$ with the global chart $h\mapsto h/\ell(h)-h_0$ onto $W$. The map $\Phi$ is bijective by (i) on every level and smooth, being the restriction of the quadratic map $g\mapsto\rho_0\,g\,(Sg)$, which is continuous on $C^\infty(X)$ by (M1) on every level; and $\Phi^{-1}$ is smooth in the convenient calculus, since a curve is smooth into $C^\infty(X)$ if and only if it is smooth into every $H^\sigma(X)$, the topology of $C^\infty(X)$ being the initial topology of the inclusions, and for a smooth curve $\tau\mapsto\rho_\tau$ of densities in $\mathcal P^\infty_+(X)$ the curves $\tau\mapsto\Phi_\sigma^{-1}(\rho_\tau)$ are smooth into $H^\sigma(X)$ by (i) and coincide by (ii).
}
\iflong
\begin{proof}
\PrfScale
\end{proof}
\else
The proof is in Appendix~\ref{app:proofs-hilbert}.
\fi

\begin{remark}[Parametrized measure models]\label{rem:ajls}
The map $U_\sigma\to\MX$, $h\mapsto\mu_h=h\,(Sh)\,\mu_0$, is a parametrized measure model in the sense of \citet[Definition~3.4]{AyJostLeSchwachhoefer2017}, dominated by $\mu_0$ with the positive density $h\,(Sh)$ and real-analytic, since $h\mapsto h\,(Sh)$ is a continuous quadratic map from $H^\sigma(X)$ into $\CX$. Since $S$ is linear, $\log\left((h+tv)\,S(h+tv)\right)=\log(h+tv)+\log(Sh+tSv)$ pointwise on $X$, so the logarithmic derivative of the model in the direction $v\in T_hU_\sigma$,
\[
\partial_v\log\left(h\,(Sh)\right)=\frac{v}{h}+\frac{Sv}{Sh},
\]
is continuous and depends continuously on $(h,v)$ in the supremum norm; by \citep[Theorem~3.2]{AyJostLeSchwachhoefer2017} the model is $k$-integrable for every $k$ \citep[Definition~3.7]{AyJostLeSchwachhoefer2017}, and its canonical tensors \citep[Section~3.2.5]{AyJostLeSchwachhoefer2017} are defined on $U_\sigma$, among them the Fisher metric
\[
g^{F}_h(v,v)=\int_X\left(\frac{v}{h}+\frac{Sv}{Sh}\right)^{2}h\,(Sh)\,d\mu_0
\]
and the Amari-Chentsov tensor \citep[Example~3.5]{AyJostLeSchwachhoefer2017}. In general these tensors live on the parameter space of a model, as pullbacks of the canonical tensors on the spaces of roots of measures \citep[Section~3.2.3]{AyJostLeSchwachhoefer2017}, directed limits over all reference measures; here every $\mu_h$ has a continuous positive density with respect to $\mu_0$, so the model takes its values in the copy of $L^{1/r}(\mu_0)$ inside these spaces, and the parameter space is the manifold of measures itself: by Theorem~\ref{thm:scale}, $h\mapsto\mu_h$ is a real-analytic diffeomorphism onto $\mathcal P^\sigma(X)$, and in the density chart $\rho$ the metric is $g^{F}_\rho(\psi,\psi)=\int_X\psi^{2}/\rho\,d\mathrm{vol}$, the Fisher-Rao metric of the positive densities of Sobolev regularity $\sigma$. The Fisher-Rao metric, the Amari-Chentsov tensor and the Sinkhorn metric of \autoref{sec:metric} are therefore tensors on one Hilbert manifold of measures, not on a parameter space mapped into it, and every model $\xi\mapsto\rho_\xi\,\mathrm{vol}$ on a manifold of parameters whose densities form a $C^{1}$ map into the positive probability densities in $H^\sigma(X)$ is a $C^{1}$ map into this manifold. By the chain rule, $\partial_V\log\rho_\xi$ is the logarithmic derivative of the stratum in the direction of the image of $V$, so the Fisher metric and the Amari-Chentsov tensor of the model are the pullbacks of those of the stratum, and the pullback of $g^{\sigma}$ is the metric tensor of \citet{lavenant_luckhardt_mordant_et_al_riemannian_2025} on the tangent measures of the model.
\end{remark}

%% file: sec4.tex
\section{The Sinkhorn metric tensor}\label{sec:metric}

In Eguchi's construction the metric of a divergence is its second-order term on the diagonal; for the Sinkhorn divergence this term is a metric tensor on $\PX$, at every $\mu$ an inner product $g_\mu$ on the signed measures of total mass zero \citep{lavenant_luckhardt_mordant_et_al_riemannian_2025}. \autoref{subsec:lavmetric} states the results of Lavenant et al.\ in the operator framework of \autoref{sec:prelim}, with references to the proofs; Sections~\ref{subsec:gradients} to~\ref{subsec:stratummetric} contain our results on the metric: its characterization through the gradients of the mean values $\mu\mapsto\int_X f\,d\mu$, the strict positivity of the correction in the sphere coordinate, and its form on the smooth stratum, where it is a weak Riemannian metric bounded by a multiple of the Fisher metric.

\subsection{The metric of Lavenant et al.}\label{subsec:lavmetric}

\subsubsection{The metric by second variation}\label{subsec:secondvar}

Following \citep[Definition~4.1]{lavenant_luckhardt_mordant_et_al_riemannian_2025}, the metric is the second-order rate of the divergence along a segment,
\begin{equation}\label{eq:gdef}
g_\mu(\delta\mu,\delta\mu)=\lim_{t\to0}\frac{\Se(\mu,\mu+t\,\delta\mu)}{t^{2}},
\end{equation}
for the perturbations $\delta\mu$ specified below. In Eguchi's convention \citep{Eguchi_1983,Eguchi_1992} the metric of a divergence $D$ is $-\partial_t\partial_s D(\mu_t,\mu_s)$ at $t=s=0$; for $\Se$ this is twice \eqref{eq:gdef}, by \eqref{eq:gF} below, and we use \eqref{eq:gdef}, the $g$ of Lavenant et al., throughout the paper.

In what follows, $\MzX=\{\nu\in\MX:\nu(X)=0\}$ denotes the \emph{balanced} measures, the finite signed measures of total mass zero, which act as tangent vectors. Fix $\mu\in\PX$ and $\delta\mu\in\MzX$ with $\delta\mu\ll\mu$ and bounded density, so that $\mu_t:=\mu+t\,\delta\mu\in\PX$ for $\vert t\vert\le t_0$ and some $t_0>0$, and define
\[
F:[-t_0,t_0]^{2}\to\mathbb R,\qquad F(t,s)=\OTe(\mu_t,\mu_s),
\]
a symmetric function, since $\OTe$ is. The two diagonal terms subtracted in the definition \eqref{eq:sinkhorndivergenz} of $\Se$ are the diagonal values of $F$,
\begin{equation}\label{eq:selfkill}
\Se(\mu,\mu_t)=F(0,t)-\tfrac12F(0,0)-\tfrac12F(t,t)=\tfrac12\left(F(0,t)+F(t,0)-F(0,0)-F(t,t)\right),
\end{equation}
the second equality by symmetry; each depends on one variable only, so every mixed derivative of $\Se(\mu_t,\mu_s)$ in $t$ and $s$ is the mixed derivative of $F$, and the metric and the connection forms of \autoref{sec:tensors} are computed thus directly from $\OTe$. If $F$ has continuous first partial derivatives and a continuous mixed partial derivative near the origin, the symmetric second difference in \eqref{eq:selfkill} has the limit \citep[Lemma~3.11]{lavenant_luckhardt_mordant_et_al_riemannian_2025}
\begin{equation}\label{eq:gF}
g_\mu(\delta\mu,\delta\mu)=-\tfrac12\,\frac{\partial^{2}F}{\partial t\,\partial s}(0,0).
\end{equation}

Write $f_{t,s}=f_{\mu_t,\mu_s}$ and $g_{t,s}=g_{\mu_t,\mu_s}$ for the Schr\"odinger potentials of the pair $(\mu_t,\mu_s)$, with the symmetric normalization of \autoref{subsec:sinkhorn} on the diagonal, $f_{t,t}=g_{t,t}=f_{\mu_t}$. The value $\OTe$ is differentiable in its marginals with the potentials as gradients \citep[Proposition~2]{feydy_sejourné_vialard_amari_trouvé_peyré_interpolating_2019}, and along the segment the envelope principle for the dual problem \eqref{eq:eot-dual} yields \citep[Lemma~3.12]{lavenant_luckhardt_mordant_et_al_riemannian_2025}
\begin{equation}\label{eq:envelope}
\frac{\partial F}{\partial t}(t,s)=\left\langle\delta\mu,\ f_{t,s}\right\rangle,
\end{equation}
a pairing that does not depend on the normalization of $f_{t,s}$, since $\delta\mu$ is balanced. The derivatives of the potentials in their marginals, with which Lavenant et al.\ iterate \eqref{eq:envelope}, are recalled in Appendix~\ref{app:thirdvar}, where the third-order computation uses them.

\begin{definition}[{\citep[Definition~4.1]{lavenant_luckhardt_mordant_et_al_riemannian_2025}}]\label{def:metric}
For $\mu\in\PX$, the \emph{Sinkhorn metric tensor} at $\mu$ is the bilinear form
\begin{equation}\label{eq:metrik}
g_\mu:\MzX\times\MzX\to\mathbb R,\qquad
g_\mu(\delta\mu_1,\delta\mu_2)
=\frac{\eps}{2}\,\left\langle\delta\mu_1,\ (\id-\Kmu^{2})^{-1}\Hmu[\delta\mu_2]\right\rangle,
\end{equation}
where $\langle\nu,\varphi\rangle=\int_X\varphi\,d\nu$.
\end{definition}

\begin{theorem}[{Second variation of $\Se$; \citep[Theorem~3.4]{lavenant_luckhardt_mordant_et_al_riemannian_2025}}]\label{thm:hessian}
Let $\mu\in\PX$, and let $\delta\mu\in\MzX$ with $\delta\mu\ll\mu$ and bounded density. Then
\begin{equation}\label{eq:hessian}
\lim_{t\to0}\frac{\Se(\mu,\mu_t)}{t^{2}}
=g_\mu(\delta\mu,\delta\mu).
\end{equation}
\end{theorem}

By \eqref{eq:selfnorm}, $\int_X\Hmu[\delta\mu]\,d\mu=\delta\mu(X)=0$, so $\Hmu[\delta\mu]$ has $\mu$-mean zero and $(\id-\Kmu^{2})^{-1}$ in \eqref{eq:metrik} acts on it as the inverse of $\id-\Kmu^{2}$ on $L^2_0(\mu)$ (the resolvent convention of \autoref{subsec:operators}). The form $g_\mu$ is defined on all balanced measures, and it is symmetric and positive definite \citep[Theorem~4.6]{lavenant_luckhardt_mordant_et_al_riemannian_2025}; Theorem~\ref{thm:hessian} identifies its quadratic form with the second-order rate \eqref{eq:gdef} of $\Se$ for the perturbations $\delta\mu\ll\mu$ with bounded density, and \citet[Theorem~3.4]{lavenant_luckhardt_mordant_et_al_riemannian_2025} state the expansion for $C^m$-perturbations, of which the segment is the case $m=0$.

\subsubsection{The metric in the sphere coordinate}\label{subsec:betametric}

\begin{theorem}[{The metric in the sphere coordinate; \citep[Definition~4.10, Theorem~4.12 and Lemma~4.13]{lavenant_luckhardt_mordant_et_al_riemannian_2025}}]\label{thm:gbeta}
Let $\mu\in\PX$ and $\beta=B(\mu)$, and let $\beta^{\perp}=\{v\in\Hc:\langle\beta,v\rangle_{\Hc}=0\}$ be the tangent space at $\beta$ of the unit sphere of $\Hc$. For $v\in\beta^{\perp}$, the function $\amu v$ lies in $L^{2}_{0}(\mu)$ by \eqref{eq:tangency}, so the symmetric bilinear form
\begin{equation}\label{eq:gbeta-bilinear}
\tilde g_\mu:\beta^{\perp}\times\beta^{\perp}\to\mathbb R,\qquad
\tilde g_\mu(v_1,v_2)=\frac{\eps}{2}\left(\langle v_1,v_2\rangle_{\Hc}+2\left\langle\amu v_1,\ (\id-\Kmu)^{-1}(\amu v_2)\right\rangle_{L^{2}(\mu)}\right)
\end{equation}
is well-defined. Its quadratic form is
\begin{equation}\label{eq:gbeta}
\tilde g_\mu(v,v)=\frac{\eps}{2}\left(\|v\|_{\Hc}^{2}+2\,Q_\mu(v)\right),
\qquad
Q_\mu(v):=\left\langle\amu v,\ (\id-\Kmu)^{-1}(\amu v)\right\rangle_{L^{2}(\mu)},
\end{equation}
whereby we call $Q_\mu$ the \emph{correction} . The form $\tilde g_\mu$ is the metric tensor in the sphere coordinate: for a weak-star continuous curve $(\mu_t)$ in $\PX$ with $\mu_0=\mu$ whose sphere image is weakly differentiable in $\Hc$ at $t=0$ with velocity $\dot\beta$, the velocity $\dot\beta$ lies in $\beta^{\perp}$, the velocity $\dot\mu$ of the curve lies in the completion $H^{\ast}_{\mu,0}=\{\xi\in\Hmu^{\ast}:\xi(1)=0\}$ of $\MzX$ for $g_\mu$, and the extension of $g_\mu$ to it satisfies $g_\mu(\dot\mu,\dot\mu)=\tilde g_\mu(\dot\beta,\dot\beta)$.
\end{theorem}

In \eqref{eq:gbeta} the first term is, up to the factor $\eps/2$, the round metric of the unit sphere of $\Hc$, which does not depend on $\mu$; the dependence on the base measure is in the correction, which is nonnegative and bounded by a multiple of the round term uniformly in $\mu$ (\autoref{subsec:positivity}).

\subsection{The metric through the gradients of mean values}\label{subsec:gradients}

\citet[Section~4]{Ay2024Otto} characterizes a Riemannian metric on probability measures by the gradients it assigns to the mean-value functions $E(f):\PX\to\mathbb R$, $E(f)(\mu)=\int_X f\,d\mu$, $f\in\CX$. The gradient of $E(f)$ at $\mu$ is the tangent vector $\operatorname{grad}_\mu E(f)$ with $g_\mu\left(\operatorname{grad}_\mu E(f),a\right)=\int_X f\,da$ for every tangent vector $a$, that is, the image of the covector of $E(f)$ under the inverse $f\mapsto f^\sharp$ of the flat map $a\mapsto a^\flat:=g_\mu(a,\cdot)$ on its range, which for a weak metric is a proper subspace of the cotangent space. We call $f\mapsto f^\sharp$ the \emph{Riesz map} of the metric, since it represents a linear form by a vector; it determines the metric, and so do the gradients of the mean-value functions \citep[Theorem~1 and Proposition~3]{Ay2024Otto}. Ay works with smooth positive densities on a closed Riemannian manifold; here the tangent space at $\mu\in\PX$ is $\MzX$ with the form $g_\mu$ of Definition~\ref{def:metric}, the cotangent vectors are taken in $\CX/\mathbb R$, the space of the covectors of the mean-value functions, which is smaller than the dual of $\MzX$ with the total variation norm, and the pairing is $\langle a,f\rangle=\int_X f\,da$, well-defined on classes modulo constants because $a$ is balanced. The construction is pointwise, does not use any differentiable structure (Remark~\ref{rem:nogo}), requires only Assumption~\ref{assump:standing} and applies at every $\mu$, so the quadratic cost is admitted. The Sinkhorn metric differs from the Fisher-Rao and Otto metrics (Remark~\ref{rem:rieszmaps}) in the domain of $\sharp$: $\Hmu[\cdot]$ is injective on measures (\autoref{subsec:rkhs}) but compact (\autoref{subsec:operators}), and its inverse exists only on its range.

\begin{lemma}[The flat map of the Sinkhorn metric]\label{lem:flat}
Let $\mu\in\PX$ and $C_0(X,\mu):=\left\{w\in\CX:\int_X w\,d\mu=0\right\}$. For $b\in\MzX$ the covector $b^\flat=g_\mu(b,\cdot)\in\CX/\mathbb R$ is represented, by Definition~\ref{def:metric}, by the function
\[
\frac{\eps}{2}\,\left(\id-\Kmu^{2}\right)^{-1}\Hmu[b],
\]
which lives in $C_0(X,\mu)$; this representative is also denoted $b^\flat$. Then: (i) the map $\MzX\to C_0(X,\mu)$, $b\mapsto b^\flat$, is linear and injective; (ii) $\id-\Kmu^{2}$ is a bijection of $C_0(X,\mu)$ onto itself; (iii) a function $w\in C_0(X,\mu)$ equals $b^\flat$ for some $b\in\MzX$ if and only if $(\id-\Kmu^{2})\,w\in\Hmu[\MX]$, and then $b=\frac{2}{\eps}\,\Hmu^{-1}\left[(\id-\Kmu^{2})\,w\right]$, a balanced measure.
\end{lemma}

\begin{proof}
The function $\Hmu[b]$ is continuous with $\mu$-mean zero, as shown after Theorem~\ref{thm:hessian}, and $(\id-\Kmu^{2})^{-1}$ maps $\CX$ into $\CX$ and $L^2_0(\mu)$ into $L^2_0(\mu)$ by the resolvent convention (\autoref{subsec:operators}), so $b^\flat\in C_0(X,\mu)$. (i) If $b^\flat=0$, then $g_\mu(b,b)=0$, so $b=0$ by the positive definiteness of $g_\mu$ (\autoref{subsec:secondvar}). (ii) $\id-\Kmu^{2}$ maps $\CX$ into $\CX$ and preserves $\mu$-means, since $\Kmu$ is self-adjoint on $L^2(\mu)$ with $\Kmu1=1$. If $w\in C_0(X,\mu)$ and $\Kmu^{2}w=w$, then $w$ is $\mu$-almost everywhere constant by Proposition~\ref{prop:operators}(ii), hence $\Kmu^{2}w$, and with it $w$, is constant everywhere, hence zero. For $u\in C_0(X,\mu)$, $w:=(\id-\Kmu^{2})^{-1}u$ lies in $C_0(X,\mu)$ and solves $(\id-\Kmu^{2})\,w=u$. (iii) By (ii), $w=b^\flat$ if and only if $\Hmu[b]=\frac{2}{\eps}(\id-\Kmu^{2})\,w$, which determines $b$ by the injectivity of $\Hmu[\cdot]$ on $\MX$ (\autoref{subsec:rkhs}). Its mass is $b(X)=\int_X\Hmu[b]\,d\mu=\frac{2}{\eps}\int_X(\id-\Kmu^{2})\,w\,d\mu=0$.
\end{proof}

\begin{definition}[Riesz map of the Sinkhorn metric]\label{def:riesz}
For $\mu\in\PX$, with the centring $\Pi_0h=h-\int_X h\,d\mu$ of \eqref{eq:meanzero}, let
\[
\mathcal D_\mu:=\left\{f\in\CX:(\id-\Kmu^{2})\,\Pi_0f\in\Hmu[\MX]\right\},
\]
a linear subspace of $\CX$ that contains the constants, and
\[
\mathcal D_\mu\to\MzX,\qquad f\mapsto f^\sharp:=\frac{2}{\eps}\,\Hmu^{-1}\left[(\id-\Kmu^{2})\,\Pi_0f\right],
\]
the measure with $(f^\sharp)^\flat=\Pi_0f$ of Lemma~\ref{lem:flat}(iii). The \emph{cotangent form} of the Sinkhorn metric is $\langle\langle f,g\rangle\rangle_\mu:=g_\mu(f^\sharp,g^\sharp)=\int_X g\,d(f^\sharp)$ for $f,g\in\mathcal D_\mu$.
\end{definition}

Since $\sharp$ vanishes on constants, it is defined on $\mathcal D_\mu/\mathbb R$, and by Lemma~\ref{lem:flat} it is a bijection of $\mathcal D_\mu/\mathbb R$ onto $\MzX$, with inverse $a\mapsto a^\flat+\mathbb R$. For infinite $X$, the domain is a dense proper subspace of $\CX$: proper because $\Hmu[\cdot]$ is compact and dense because $\Hmu[\MX]$ contains the sections $k_\mu(x,\cdot)$, whose span is dense in $\Hmu$, and $\Hmu$ is itself dense in $\CX$ by Lemma~\ref{lem:tilt-isometry} (and the universality in Assumption~\ref{assump:standing}). %In the two classical cases, the corresponding domains are all continuous functions for Fisher-Rao and the twice differentiable ones for Otto, a restriction that Ay's smooth setting does not show.

\begin{lemma}[Functional calculus for $\Kmu$ on $L^2_0(\mu)$]\label{lem:spectral}
Let $\mu\in\PX$. The operator $\Kmu$ maps $L^2_0(\mu)$ into itself and is compact, self-adjoint and injective there, so $L^2_0(\mu)$ has an orthonormal basis $(e_j)_{j\in J}$ of eigenvectors, $\Kmu e_j=\lambda_je_j$ with $\lambda_j\in(0,q]$. For $\psi\in L^2_0(\mu)$ write $c_j=\langle\psi,e_j\rangle_{L^2(\mu)}$, and for a bounded function $\phi:(0,q]\to\mathbb R$, let $\phi(\Kmu)\psi:=\sum_{j\in J}\phi(\lambda_j)\,c_j\,e_j$, an operator on $L^2_0(\mu)$.
(i) For a rational function $\phi=p_1/p_2$ whose denominator does not vanish on $[0,q]$, $\phi(\Kmu)=p_1(\Kmu)\,p_2(\Kmu)^{-1}$; in particular $(\id\pm\Kmu)^{-1}$, $(\id-\Kmu^{2})^{-1}$ and $\Kmu(\id-\Kmu^{2})^{-1}$ are of this form. The inverse $\Kmu^{-1}$ is defined on $\Kmu(L^2_0(\mu))=\{\psi\in L^2_0(\mu):\sum_jc_j^{2}/\lambda_j^{2}<\infty\}$ and acts by $e_j\mapsto\lambda_j^{-1}e_j$.
(ii) $\langle\psi,\phi(\Kmu)\psi\rangle_{L^2(\mu)}=\sum_{j}\phi(\lambda_j)\,c_j^{2}$; hence, if $\phi_2>0$ on $(0,q]$ and $\phi_1/\phi_2$ is nondecreasing there, then $\langle\psi,\phi_1(\Kmu)\psi\rangle_{L^2(\mu)}\le\frac{\phi_1(q)}{\phi_2(q)}\,\langle\psi,\phi_2(\Kmu)\psi\rangle_{L^2(\mu)}$.
\end{lemma}

\begin{proof}
By Proposition~\ref{prop:operators}(ii), $\Kmu$ is compact and self-adjoint on $L^2(\mu)$ with $\Kmu1=1$, so it preserves $L^2_0(\mu)$, whereby its spectrum lies in $[0,q]$; it is injective there, since $\langle\Kmu\psi,\psi\rangle_{L^2(\mu)}=\|\Hmu[\psi\mu]\|_{\Hmu}^{2}$ by \eqref{eq:mmd-norm} for the self-kernel and $\Hmu[\cdot]$ is injective on measures (\autoref{subsec:rkhs}). The spectral theorem yields the eigenbasis, and (i) and (ii) follow by acting on it and by Parseval's identity.
\end{proof}

\begin{lemma}[The norm of $\Hmu$ in the eigenbasis]\label{lem:hmunorm}
Let $\mu\in\PX$, with $(e_j)_{j\in J}$, $\lambda_j$ and $c_j$ as in Lemma~\ref{lem:spectral}. For $\varphi\in L^2(\mu)$, $\|\Kmu\varphi\|_{\Hmu}^{2}=\langle\varphi,\Kmu\varphi\rangle_{L^2(\mu)}$; hence $\|\psi\|_{\Hmu}^{2}=\sum_jc_j^{2}/\lambda_j$ for $\psi\in\Kmu(L^2_0(\mu))$. The mean-zero subspace of $\Hmu$ is the $\Hmu$-orthogonal sum of the closure of $\Kmu(L^2_0(\mu))$ in $\Hmu$ and of $\{\varphi\in\Hmu:\varphi=0\ \mu\text{-a.e.}\}$; the second summand is trivial if $\mu$ has full support.
\end{lemma}

\begin{proof}
The duality \eqref{eq:kme-duality} for $k_\mu$ yields $\langle\Kmu\varphi,\chi\rangle_{\Hmu}=\int_X\chi\varphi\,d\mu$ for $\chi\in\Hmu$; with $\chi=\Kmu\varphi$ and $\varphi=\sum_jd_je_j$ this is $\sum_j\lambda_jd_j^{2}=\sum_jc_j^{2}/\lambda_j$, and the same duality shows that a $\varphi$ of $\mu$-mean zero in $\Hmu$ is orthogonal to $\Kmu(L^2_0(\mu))$ exactly when it vanishes $\mu$-a.e., hence everywhere if $\operatorname{supp}\mu=X$.
\end{proof}

On measures with a density, the Riesz map and the cotangent form become explicit. If $\Pi_0f=\Kmu\psi$ in $C(X)$ with $\psi\in L^2_0(\mu)$, since $\Hmu[\psi\mu]=\Kmu\psi$ by \eqref{eq:Kmu}, we obtain
\begin{equation}\label{eq:riesz-density}
f^\sharp=\frac{2}{\eps}\left(\Kmu^{-1}-\Kmu\right)\Pi_0f\cdot\mu,
\qquad
\langle\langle f,g\rangle\rangle_\mu=\frac{2}{\eps}\left\langle\left(\Kmu^{-1}-\Kmu\right)\Pi_0f,\ \Pi_0g\right\rangle_{L^2(\mu)},
\end{equation}
for $f,g$ with $\Pi_0f,\Pi_0g\in\Kmu\left(L^2_0(\mu)\right)$, where $\Kmu^{-1}\Pi_0f$ denotes the unique $\psi\in L^2_0(\mu)$ with $\Kmu\psi=\Pi_0f$ in $C(X)$ (Lemma~\ref{lem:spectral}(i)). By \eqref{eq:selfnorm}, for $w\in L^2(\mu)$,
\begin{equation}\label{eq:dirichlet}
\left\langle(\id-\Kmu)\,w,w\right\rangle_{L^2(\mu)}
=\frac12\int_X\int_X\left(w(x)-w(y)\right)^{2}k_\mu(x,y)\,d\mu(x)\,d\mu(y)
=:\mathcal E_\mu(w,w),
\end{equation}
the Dirichlet form of the reversible chain of \autoref{subsec:operators} with transition density $k_\mu$ and invariant measure $\mu$. In the eigenbasis of Lemma~\ref{lem:spectral}, with $c_j$ the coefficients of $\Pi_0f$,
\[
\langle\langle f,f\rangle\rangle_\mu=\frac{2}{\eps}\sum_j\frac{1+\lambda_j}{\lambda_j}\,(1-\lambda_j)\,c_j^{2},
\qquad
\mathcal E_\mu(\Pi_0f,\Pi_0f)=\sum_j(1-\lambda_j)\,c_j^{2}.
\]
So the cotangent form of the Sinkhorn metric is a spectral reweighting of the Dirichlet form of the Sinkhorn chain, with the weight $(1+\lambda)/\lambda$ on the eigenvalues $\lambda\in(0,q]$, and since $(1+\lambda)/\lambda\ge2$ for $0<\lambda\le1$, $\langle\langle f,f\rangle\rangle_\mu\ge\frac{4}{\eps}\,\mathcal E_\mu(\Pi_0f,\Pi_0f)$. The reweighted form is not a Dirichlet form in general: its domain is the mean-zero part of $\Hmu$ (Lemma~\ref{lem:hmunorm}), and at the uniform measure on the circle with the Mat\'ern kernel of \autoref{sec:stratum} and $s\ge2$ this is the mean-zero part of $H^{s}(\mathbb T)$, which $f\mapsto|f|$ does not preserve. The Otto metric has in this place the Dirichlet form $\int_X|\nabla w|^{2}\,d\mu$ of the diffusion with generator $\Delta_\mu=\Delta+\langle\nabla\log\rho,\nabla\cdot\rangle$, Brownian motion with drift $\nabla\log\rho$.

\begin{theorem}[Gradients of mean values]\label{thm:gradients}
Let $\mu\in\PX$.
(i) For $f\in\mathcal D_\mu$ there is exactly one $a\in\MzX$ with $g_\mu(a,b)=\int_X f\,db$ for all $b\in\MzX$, namely $a=f^\sharp$. It is the gradient $\operatorname{grad}_\mu E(f)$ of the mean-value function with respect to the Sinkhorn metric.
(ii) For $f\in\CX\setminus\mathcal D_\mu$ no such $a$ exists.
(iii) Every $a\in\MzX$ is a gradient, $a=\operatorname{grad}_\mu E(a^\flat)$. Consequently, a bilinear form $g'$ on $\MzX$ with $g'(f^\sharp,b)=\int_X f\,db$ for all $f\in\mathcal D_\mu$ and $b\in\MzX$ equals $g_\mu$.
\end{theorem}

\begin{proof}
By Lemma~\ref{lem:flat} and the symmetry of $g_\mu$, $g_\mu(a,b)=\langle b,a^\flat\rangle$. Hence $g_\mu(a,\cdot)=\int_X f\,d(\cdot)$ on $\MzX$ if and only if $\langle b,a^\flat-f\rangle=0$ for all $b\in\MzX$, if and only if $a^\flat-f$ is constant, by the choice $b=\delta_x-\delta_y$, if and only if $a^\flat=\Pi_0f$, since $a^\flat$ has $\mu$-mean zero. By Lemma~\ref{lem:flat}(iii) such an $a$ exists if and only if $(\id-\Kmu^{2})\,\Pi_0f\in\Hmu[\MX]$, that is $f\in\mathcal D_\mu$, and then $a=f^\sharp$; this is (i) and (ii). For (iii), $f:=a^\flat$ lies in $C_0(X,\mu)$ with $(\id-\Kmu^{2})f=\frac{\eps}{2}\Hmu[a]\in\Hmu[\MX]$, so $f\in\mathcal D_\mu$ and $f^\sharp=a$; and for $a,b\in\MzX$, $g'(a,b)=g'\left((a^\flat)^\sharp,b\right)=\int_X a^\flat\,db=g_\mu(a,b)$.
\end{proof}

\begin{remark}[The FR, Otto and Sinkhorn Riesz maps]\label{rem:rieszmaps}
For the Fisher-Rao metric, $f^\sharp=\Pi_0f\,\mu$ for every $f\in\CX$ and every $\mu$, and the cotangent form is the variance $\int_X(\Pi_0f)^{2}\,d\mu$. For the Otto metric on a closed Riemannian manifold and $\mu$ with a smooth positive density $\rho$, $f^\sharp=-\Delta_\mu f\,\mu$ for $f\in C^{2}$, with $\Delta_\mu f=\Delta f+\langle\nabla\log\rho,\nabla f\rangle$ \citep[Theorem~1]{Ay2024Otto}, and the cotangent form is $\int_X|\nabla f|^{2}\,d\mu$, the Dirichlet form of the diffusion with generator $\Delta_\mu$, which is reversible for $\mu$. For the Sinkhorn metric, under Assumption~\ref{assump:standing} and for every $\mu$, $f^\sharp=\frac{2}{\eps}(\Kmu^{-1}-\Kmu)\,\Pi_0f\cdot\mu$ for $f$ with $\Pi_0f\in\Kmu(L^2_0(\mu))$, and the cotangent form is the spectral reweighting \eqref{eq:dirichlet} of the Dirichlet form of the Sinkhorn chain. So the three metrics \textit{are of the same kind}: their cotangent forms are built from a process reversible for $\mu$, the Dirichlet forms of independent resampling from $\mu$ and of the diffusion with generator $\Delta_\mu$, and a spectral reweighting of the Dirichlet form of the Sinkhorn chain. \citet[Definition~1]{AySchwachhoefer2025Torsion} write a metric of this kind as the Fisher-Rao metric twisted by an operator, $g_\mu(\psi\mu,\chi\mu)=\langle\Theta_\mu\psi,\chi\rangle_{L^2(\mu)}$, with $\Theta_\mu=\id$ for Fisher-Rao and $\Theta_\mu=(-\Delta_\mu)^{-1}$ for Otto. For the Sinkhorn metric,
\begin{equation}\label{eq:Theta}
\Theta_\mu=\frac{\eps}{2}\left(\Kmu^{-1}-\Kmu\right)^{-1}=\frac{\eps}{2}\,\Kmu\left(\id-\Kmu^{2}\right)^{-1}=\frac{\eps}{2}\sum_{j\ge0}\Kmu^{2j+1}
\quad\text{on }L^2_0(\mu),
\end{equation}
the series converging since $q<1$; this is Theorem~\ref{thm:hessian} with $\Hmu[\psi\mu]=\Kmu\psi$. On the torus with the Mat\'ern kernel and $\mu$ with a smooth positive density, $\Kmu=M_{\amu}\,S\,M_{\amu\rho/\rho_0}$, with $\rho_0$ the density of $\mu_0$ with respect to the volume, is a bijection of the smooth functions of $\mu$-mean zero, since $S$ is one and $\amu$ is smooth, so the Sinkhorn metric is a regular metric in their sense, and their $\alpha$-connections, the mixture connection and its dual, are defined (Remark~\ref{rem:selfdual}).
\end{remark}

\begin{remark}[The limits in $\eps$]\label{rem:epslimits}
The two limits concern two different families of kernels. As $\eps\to0$ the cost is held fixed. In the case of $X=\mathbb T^{d}$ with the periodized Gaussian kernel and of $\mu=\rho\,dx$ with $\rho$ smooth and positive, the self-potential has the Laplace expansion $\fmu=-\frac{\eps}{2}\log\rho+\text{const}+O(\eps^{2})$ \citep[Theorem~1]{Pal2019EntropicOTDiff}, \citep[Theorem~1.6]{ConfortiTamanini2021}. Formally, this expansion yields $\frac{2}{\eps}(\Kmu^{-1}-\Kmu)\to-\Delta_\mu$, so the Riesz map and the cotangent form of the Sinkhorn metric tend to those of the Otto metric. The sphere embedding $B(\mu)=\amu^{-1}$, normalized in $L^{2}(dx)$, tends to the Hellinger embedding $\mu\mapsto\sqrt\rho$. The limits are formal because the expansion $\Kmu=\id+\frac{\eps}{4}\Delta_\mu+o(\eps)$ of \citet[Section~4.4]{lavenant_luckhardt_mordant_et_al_riemannian_2025} does not extend to $\Kmu^{-1}$, whose domain is a proper subspace of the smooth functions for every $\eps$. At the uniform measure, the Otto limit holds exactly on trigonometric polynomials, which lie in $\mathcal D_\mu$ for every $\eps$: the self-potential is constant, the eigenvalue of $\Kmu$ on $\cos(2\pi n\cdot x)$ is $e^{-\pi^{2}\eps|n|^{2}}$, and the Riesz map acts on this cosine by $\frac{4}{\eps}\sinh(\pi^{2}\eps|n|^{2})$, which tends to $4\pi^{2}|n|^{2}$, the eigenvalue of $-\Delta$, as $\eps\to0$, and to infinity as $\eps\to\infty$. As $\eps\to\infty$ it is the kernel that varies, $k_c=e^{-c/\eps}$ for a fixed continuous cost $c$ with Assumption~\ref{assump:standing} for every $\eps$: since $0\le\fmu\le\|c\|_\infty$ by \eqref{eq:selfnorm} and $k_c=1-c/\eps+O(\eps^{-2})$, $\Kmu=-\frac{1}{\eps}\,\Pi_0C_\mu\Pi_0+O(\eps^{-2})$ in operator norm on $L^2_0(\mu)$, uniformly in $\mu$, where $(C_\mu\psi)(y)=\int_X c(x,y)\,\psi(x)\,d\mu(x)$, so $\Theta_\mu\to-\frac12\Pi_0C_\mu\Pi_0$ and the metric tensor converges on density perturbations to the constant form $-\frac12\int_X\int_X c\,d\dot\mu\,d\dot\mu$ of the squared maximum mean discrepancy with kernel $-c/2$, the limit of $\Se$ \citep{RamdasGarciaTrillosCuturi2017,feydy_sejourné_vialard_amari_trouvé_peyré_interpolating_2019}. This form may be degenerate, (for the quadratic cost it is $\left|\int_X x\,d\dot\mu(x)\right|^{2}$ \citep[Theorem~4.17]{lavenant_luckhardt_mordant_et_al_riemannian_2025}), so the Riesz maps do not have a limit in general; for a fixed Mat\'ern kernel of \autoref{sec:stratum} the operators do not depend on $\eps$, and $g_\mu$ is linear in $\eps$.
\end{remark}

\subsection{Comparison with the round metric of the unit sphere}\label{subsec:positivity}

The correction $Q_\mu$ in \eqref{eq:gbeta} is strictly positive on measures of full support, and \citet[Proposition~4.15]{lavenant_luckhardt_mordant_et_al_riemannian_2025} bound it by a multiple of the round term uniformly in $\mu$; we obtain their bound with an improved constant.

\begin{proposition}[Strict positivity of the correction]\label{prop:positivity}
Let $\mu\in\PX$, $\beta=B(\mu)$ and $v\in\beta^{\perp}$. Then
\[
Q_\mu(v)\ \ge\ \|\amu v\|_{L^{2}(\mu)}^{2},
\]
with equality if and only if $v=0$ $\mu$-almost everywhere, and $Q_\mu(v)=0$ in that case. In particular, if $\mu$ has full support, then $Q_\mu(v)>\|\amu v\|_{L^{2}(\mu)}^{2}>0$ for every $v\in\beta^{\perp}$ with $v\ne0$.
\end{proposition}

\begin{proof}
Let $w=\amu v$; by \eqref{eq:tangency}, $w\in L^{2}_{0}(\mu)$, and $R=(\id-\Kmu)^{-1}$ satisfies $R\ge\id$ there (Proposition~\ref{prop:operators}(iii)), so $Q_\mu(v)=\langle w,Rw\rangle_{L^{2}(\mu)}\ge\|w\|_{L^{2}(\mu)}^{2}$. For the equality case, $R-\id=\Kmu R$, and since $R$ is positive and commutes with $\Kmu$, $\langle w,\Kmu Rw\rangle_{L^{2}(\mu)}=\langle R^{1/2}w,\Kmu R^{1/2}w\rangle_{L^{2}(\mu)}$ vanishes if and only if $\Kmu R^{1/2}w=0$, that is, $\Kmu w=0$; and $\Kmu$ is injective on $L^{2}_{0}(\mu)$ (Lemma~\ref{lem:spectral}), so equality holds if and only if $w=0$ in $L^{2}(\mu)$, which is $v=0$ $\mu$-almost everywhere because $\amu>0$. If $\operatorname{supp}\mu=X$, a function of $\Hc\subset\CX$ that vanishes $\mu$-almost everywhere vanishes on a dense set, hence everywhere.
\end{proof}

The condition on the support cannot be dropped: for $\mu=\delta_x$ the space $L^{2}_{0}(\delta_x)$ is trivial, $Q_{\delta_x}\equiv0$, the tangent space $\beta^{\perp}$ consists of the $v\in\Hc$ with $v(x)=0$ by Example~\ref{ex:dirac}, and $\tilde g_{\delta_x}(v,v)=\frac{\eps}{2}\|v\|_{\Hc}^{2}$ is the round metric alone.

\begin{proposition}[{Uniform equivalence with the round metric; \citep[Proposition~4.15]{lavenant_luckhardt_mordant_et_al_riemannian_2025}}]\label{prop:sandwich}
Let $\mu\in\PX$, $\beta=B(\mu)$, and let $q=1-e^{-\|c\|_\infty/\eps}$ be the constant of Proposition~\ref{prop:operators}. For every $v\in\beta^{\perp}$,
\begin{equation}\label{eq:aequivalenz}
\frac{\eps}{2}\,\|v\|_{\Hc}^{2}
\ \le\ \tilde g_\mu(v,v)
\ \le\ C\,\frac{\eps}{2}\,\|v\|_{\Hc}^{2},
\qquad
C=\frac{1+q}{1-q}=2\,e^{\|c\|_\infty/\eps}-1,
\end{equation}
equivalently $0\le Q_\mu(v)\le\frac{q}{1-q}\|v\|_{\Hc}^{2}$, with $\frac{q}{1-q}=e^{\|c\|_\infty/\eps}-1$. The constant $C$ does not depend on $\mu$.
\end{proposition}

\begin{proof}
Let $w=\amu v$ and $R=(\id-\Kmu)^{-1}$. By Lemma~\ref{lem:tilt-isometry}, $w\in\Hmu$ with $\|w\|_{\Hmu}=\|v\|_{\Hc}$, so the round term is $\|w\|_{\Hmu}^{2}$, and we write $Q_\mu$ in the inner product of $\Hmu$ to compare the two. The function $w$ has mean zero under $\mu$ by \eqref{eq:tangency}, so it lies in the mean-zero subspace of $\Hmu$, and so does $Rw$, by Proposition~\ref{prop:operators}(iii). Proposition~\ref{prop:operators}(i) and the self-adjointness of $\Kmu$ on $\Hmu$ yield
\[
Q_\mu(v)=\left\langle w,Rw\right\rangle_{L^{2}(\mu)}=\left\langle\Kmu w,Rw\right\rangle_{\Hmu}=\left\langle w,\Kmu Rw\right\rangle_{\Hmu}.
\]
On the mean-zero subspace of $\Hmu$ the spectrum of $\Kmu$ lies in $[0,q]$ by Proposition~\ref{prop:operators}(ii), and $\lambda\mapsto\lambda/(1-\lambda)$ is increasing on $[0,q]$. By the spectral theorem, $0\le Q_\mu(v)\le\frac{q}{1-q}\|w\|_{\Hmu}^{2}=\frac{q}{1-q}\|v\|_{\Hc}^{2}$. Inserting this into \eqref{eq:gbeta} and using $1+\frac{2q}{1-q}=\frac{1+q}{1-q}$ yields \eqref{eq:aequivalenz}.
\end{proof}

\citet[Proposition~4.15]{lavenant_luckhardt_mordant_et_al_riemannian_2025} argue in $L^{2}(\mu)$ instead, with their constants, and obtain $C=1+2e^{11\|c\|_\infty/(2\eps)}$. The constant $C$ depends on $\eps$ and the cost only through the lower bound $1-q=e^{-\|c\|_\infty/\eps}$ of Proposition~\ref{prop:operators} for the spectral gap of the chain: it tends to $1$ as $\eps\to\infty$, where the correction becomes small against the round term, and grows like $2e^{\|c\|_\infty/\eps}$ as $\eps\to0$, where the gap closes.

\subsection{The weak Riemannian metric on the smooth stratum}\label{subsec:stratummetric}

In the parametrization $h$ of the smooth stratum, with Assumptions~\ref{assump:standing} and~\ref{assump:stratum} in place, the metric tensor of Definition~\ref{def:metric} becomes an explicit bilinear form on $T_hU_\sigma$, with the tilt, the transition operator and the resolvent expressed through $h$ and the operator $S$ of \eqref{eq:Sop} (Proposition~\ref{prop:stratumflat}). Throughout, $h\in U_\sigma$, $\mu_h$ is the measure \eqref{eq:hchart}, $a=a_{\mu_h}=1/(Sh)$, $B(\mu_h)=Sh$ and $T_hU_\sigma=\{v\in H^{\sigma}(X):\langle v,Sh\rangle_0=0\}$ by Proposition~\ref{prop:stratum}, and $R=(\id-K_{\mu_h})^{-1}$ on $L^{2}_{0}(\mu_h)$ in the resolvent convention. Since $d\mu_h=h\,(Sh)\,d\mu_0$ with a continuous positive density, the spaces $L^{2}(\mu_h)$ and $L^{2}(\mu_0)$ coincide as sets, with equivalent norms; their inner products differ by the weight $h\,(Sh)$, and $\langle\cdot,\cdot\rangle_0$ is that of $L^{2}(\mu_0)$. A curve $(h_t)$ in $U_\sigma$ differentiable in $H^{\sigma}(X)$ at $t=0$ with velocity $v$ has the curve of measures $(\mu_{h_t})$, differentiable in total variation with velocity
\begin{equation}\label{eq:tangentialmass}
\dot\mu=\left(v\,Sh+h\,Sv\right)\mu_0=\psi_v\,\mu_h,
\qquad
\psi_v:=\frac{v}{h}+\frac{Sv}{Sh},
\end{equation}
a balanced measure with the continuous density $\psi_v$ with respect to $\mu_h$, the logarithmic derivative of the density $h\,(Sh)$ of Remark~\ref{rem:ajls}. The metric tensor $g_{\mu_h}$ of Definition~\ref{def:metric} applies to such measures, and the stratum metric is its pullback,
\begin{equation}\label{eq:gsigma}
g^{\sigma}_{h}:T_hU_\sigma\times T_hU_\sigma\to\mathbb R,
\qquad
g^{\sigma}_{h}(v_1,v_2):=g_{\mu_h}\left(\psi_{v_1}\mu_h,\ \psi_{v_2}\mu_h\right),
\end{equation}
whose quadratic form is the second-order rate \eqref{eq:gdef} of $\Se$ along the segment through $\mu_h$ with velocity $\dot\mu$ (Theorem~\ref{thm:hessian}). The link with the sphere coordinate is the identity
\begin{equation}\label{eq:bridge}
H_{\mu_h}[\dot\mu]=\left(\id+K_{\mu_h}\right)(a\,Sv),
\end{equation}
which follows from $k_{\mu_h}(x,y)=a(x)\,k_c(x,y)\,a(y)$ and \eqref{eq:faktorisierung}:
\[
\begin{aligned}
H_{\mu_h}[\dot\mu](y)&=a(y)\int_X k_c(x,y)\left(v(x)+a(x)\,h(x)\,(Sv)(x)\right)d\mu_0(x)\\
&=a\,Sv+a\,S(h\,a\,Sv).
\end{aligned}
\]
The function $a\,Sv$ has $\mu_h$-mean zero, since $\int_X a\,Sv\,d\mu_h=\langle Sv,h\rangle_0=\langle v,Sh\rangle_0=0$, which is \eqref{eq:tangency} for $w=Sv$: the sphere velocity is tangent (Proposition~\ref{prop:stratum}(iii)).

\begin{proposition}[The stratum metric in the chart $h$]\label{prop:stratumflat}
Let $h\in U_\sigma$, let $M_g$ denote multiplication by a function $g$, and define
\begin{equation}\label{eq:Th}
T_h:T_hU_\sigma\to L^{2}(\mu_0),
\qquad
T_h:=\id+2\,M_{h}\,R\,M_{a}\,S.
\end{equation}
Then, for $v,v_1,v_2\in T_hU_\sigma$:
(i) with $\Theta_{\mu_h}$ as in \eqref{eq:Theta} and $\tilde g_{\mu_h}$ as in \eqref{eq:gbeta-bilinear},
\begin{equation}\label{eq:stratum-metric}
g^{\sigma}_{h}(v_1,v_2)=\left\langle\Theta_{\mu_h}\psi_{v_1},\psi_{v_2}\right\rangle_{L^{2}(\mu_h)}=\tilde g_{\mu_h}(Sv_1,Sv_2)=\frac{\eps}{2}\,\left\langle v_1,\ S\,T_h\,v_2\right\rangle_0;
\end{equation}
(ii) with $C=(1+q)/(1-q)$ as in Proposition~\ref{prop:sandwich},
\begin{equation}\label{eq:stratum-sandwich}
\frac{\eps}{2}\,\langle v,Sv\rangle_0
\ \le\ g^{\sigma}_{h}(v,v)
\ \le\ C\,\frac{\eps}{2}\,\langle v,Sv\rangle_0,
\end{equation}
where the first inequality is strict for $v\ne0$; in particular $g^{\sigma}_h$ is a positive definite symmetric bilinear form on $T_hU_\sigma$.
\end{proposition}

\begin{proof}
For (i), the first equality is \eqref{eq:Theta} with $H_{\mu_h}[\psi\mu_h]=K_{\mu_h}\psi$ for bounded $\psi$, \eqref{eq:Kmu}, since $\psi_{v_i}$ is continuous with $\mu_h$-mean zero. For the expression $\frac{\eps}{2}\langle v_1,S\,T_hv_2\rangle_0$, \eqref{eq:bridge} and the factorization $\id-K_{\mu_h}^{2}=(\id-K_{\mu_h})(\id+K_{\mu_h})$ on $L^{2}_{0}(\mu_h)$ yield $(\id-K_{\mu_h}^{2})^{-1}H_{\mu_h}[\dot\mu_2]=R(a\,Sv_2)$, so by \eqref{eq:metrik} and \eqref{eq:tangentialmass}
\[
g^{\sigma}_{h}(v_1,v_2)=\frac{\eps}{2}\int_X R(a\,Sv_2)\left(v_1\,Sh+h\,Sv_1\right)d\mu_0 .
\]
Since $R=\id+K_{\mu_h}R$ and $(Sh)\,K_{\mu_h}\varphi=S(h\varphi)$ by \eqref{eq:faktorisierung}, $(Sh)\,R(a\,Sv_2)=Sv_2+S\left(h\,R(a\,Sv_2)\right)$, and by the self-adjointness of $S$ on $L^{2}(\mu_0)$ the integral equals
\[
\begin{aligned}
&\langle v_1,Sv_2\rangle_0+\left\langle v_1,S\left(h\,R(a\,Sv_2)\right)\right\rangle_0+\left\langle Sv_1,h\,R(a\,Sv_2)\right\rangle_0\\
&\qquad=\langle v_1,Sv_2\rangle_0+2\left\langle v_1,\ S\,M_h\,R\,M_a\,S\,v_2\right\rangle_0
=\langle v_1,S\,T_hv_2\rangle_0 ,
\end{aligned}
\]
so $g^{\sigma}_{h}(v_1,v_2)=\frac{\eps}{2}\langle v_1,S\,T_hv_2\rangle_0$, the last expression in \eqref{eq:stratum-metric}. It equals the middle expression $\tilde g_{\mu_h}(Sv_1,Sv_2)$, since the round term of \eqref{eq:gbeta-bilinear} is $\langle Sv_1,Sv_2\rangle_{\Hc}=\langle v_1,Sv_2\rangle_0$ by the duality \eqref{eq:kme-duality}, and its correction is
\[
\langle a\,Sv_1,R(a\,Sv_2)\rangle_{L^{2}(\mu_h)}=\int_X(Sv_1)\,h\,R(a\,Sv_2)\,d\mu_0=\langle v_1,SM_hRM_aSv_2\rangle_0,
\]
where $a\,Sv_2\in L^{2}_{0}(\mu_h)$ because $\int_X a\,Sv_2\,d\mu_h=\langle Sv_2,h\rangle_0=\langle v_2,Sh\rangle_0=0$. For (ii), $Sv\in(Sh)^{\perp}$ since $\langle Sv,Sh\rangle_{\Hc}=\langle v,Sh\rangle_0=0$; apply Proposition~\ref{prop:sandwich} at $\mu_h$ to $Sv$ and the identity $g^{\sigma}_h(v,v)=\tilde g_{\mu_h}(Sv,Sv)$ of (i), with the identity for the round term. If $v\ne0$ then $Sv=\Hc[v\mu_0]\ne0$ by the injectivity of the kernel mean embedding (\autoref{subsec:rkhs}), and since $Sv$ is continuous and $\mu_h$ has full support, $Sv$ is nonzero on a set of positive $\mu_h$-measure. By Proposition~\ref{prop:positivity} the correction is then strictly positive, so the first inequality is strict, and $g^{\sigma}_{h}(v,v)>\frac{\eps}{2}\langle v,Sv\rangle_0\ge0$.
\end{proof}

\begin{proposition}[The stratum metric is weak]\label{prop:schwach}
Let $h\in U_\sigma$. The flat map $\flat_h:T_hU_\sigma\to(T_hU_\sigma)^{\ast}$, $v\mapsto g^{\sigma}_{h}(v,\cdot)$, is injective and compact; consequently it is not surjective, and $g^{\sigma}$ is a weak Riemannian metric on $U_\sigma$.
\end{proposition}

\begin{proof}
The operator $T_h$ of \eqref{eq:Th} maps $T_hU_\sigma$ boundedly into $L^{2}(\mu_0)$, since $S$ maps $H^{\sigma}(X)$ into $\CX$ by (M1), $a$ and $h$ are continuous and $R$ is bounded (Proposition~\ref{prop:operators}(iii)); so $S\,T_h$ maps $T_hU_\sigma$ boundedly into $H^{2s}(X)$, by (M1) with $r=0$. Hence, by \eqref{eq:stratum-metric}, $\flat_h$ is $\frac{\eps}{2}$ times the bounded map $S\,T_h:T_hU_\sigma\to H^{2s}(X)$ followed by the map $H^{2s}(X)\to(T_hU_\sigma)^{\ast}$, $\psi\mapsto\langle\cdot,\psi\rangle_0$, which factors through the compact embedding $H^{2s}(X)\hookrightarrow H^{-\sigma}(X)$ on the closed manifold $X$, so $\flat_h$ is compact. It is injective by the strict inequality in Proposition~\ref{prop:stratumflat}(ii), and it is not surjective: by the open mapping theorem a compact surjection onto the infinite-dimensional space $(T_hU_\sigma)^{\ast}$ would make a ball of that space relatively compact, which Riesz's theorem excludes.
\end{proof}

By Propositions~\ref{prop:stratumflat}(ii) and~\ref{prop:schwach}, $g^{\sigma}$ is weak in a quantitative sense. The norm it induces on $T_hU_\sigma$ is $\langle v,Sv\rangle_0^{1/2}=\|S^{1/2}v\|_0$, equivalent to the norm of $H^{-s}(X)$ (the paragraph after \eqref{eq:sobolevskala}), whereas the manifold topology of $U_\sigma$ is that of $H^{\sigma}(X)$. Correspondingly, a general element of $(T_hU_\sigma)^{\ast}$ has the form $u\mapsto\langle u,w\rangle_0=\int_X u\,w\,d\mu_0$ with $w\in H^{-\sigma}(X)$, the pairing extended by duality, whereas for $\flat_h(v)$, by \eqref{eq:stratum-metric}, $w=\frac{\eps}{2}\,S\,T_hv$, which lies in $H^{\sigma+2s}(X)$ (Proposition~\ref{prop:range}(ii)). Only covectors in the range of $\flat_h$ have gradients, and the Koszul formula defines the Levi-Civita connection only where its right-hand side lies in that range \citep[Section~2]{BauerBruverisMichor2014}, which is why Proposition~\ref{prop:range} determines the range exactly.

The metric is compatible with the scale of Theorem~\ref{thm:scale}: for $\sigma'>\sigma$, $g^{\sigma'}_h$ is the restriction of $g^{\sigma}_h$ to $T_hU_{\sigma'}\subset T_hU_\sigma$, both being the pullback \eqref{eq:gsigma} of $g_{\mu_h}$ with the same tangent measures \eqref{eq:tangentialmass}, and if Assumption~\ref{assump:stratum} holds for every $\sigma>d/2$, the restriction of these forms to $U_\infty$, where they are all the same, is a weak Riemannian metric on Ay's manifold $\mathcal P^\infty_+(X)$, smooth because it is real-analytic on every stratum. The constants in \eqref{eq:stratum-sandwich} depend neither on $h$ nor on $\sigma$: on every stratum the norm induced on the tangent spaces is that of $H^{-s}(X)$, while the manifold topology of $U_\sigma$ is that of $H^{\sigma}(X)$. In the sphere coordinate, the same tensor is uniformly equivalent to the round metric of $\Hc$ (Proposition~\ref{prop:sandwich}); the two topologies on the same set $\mathcal P^{\sigma}(X)$ differ because $S:H^{\sigma}(X)\to\Hc$ is compact (the paragraph after \eqref{eq:faktorisierung}), and uniform estimates (in the base point) come from the sphere coordinate through \eqref{eq:aequivalenz}.

The range of $\flat_h$, on which gradients and the Levi-Civita connection depend, is explicit. In the density chart of Theorem~\ref{thm:scale}(i), the tangent space of the stratum at $\mu=\mu_h$ consists of the measures $\varphi\mu$ with $\varphi\in H^{\sigma}(X)$ of $\mu$-mean zero, and on them the metric is the form $\langle\Theta_\mu\varphi,\chi\rangle_{L^2(\mu)}$ of \eqref{eq:Theta}, so the range is determined by the action of $\Kmu=M_aSM_h$, \eqref{eq:faktorisierung}, on the Sobolev scale; in the parametrization, \eqref{eq:stratum-metric} exhibits Arnold's inertia operator $\tfrac{\eps}{2}\,S\,T_h$, the fixed smoothing operator $S$ composed with $T_h$; it sends a tangent vector to the representative of its covector and maps $T_hU_\sigma$ onto $\{w\in H^{\sigma+2s}(X):\langle h,w\rangle_0=0\}$, since $T_h$ is a bijection of $T_hU_\sigma$ (Proposition~\ref{prop:range}(ii)).

\begin{lemma}[The transition operator on the Sobolev scale]\label{lem:sobolevK}
Let Assumptions~\ref{assump:standing} and~\ref{assump:stratum} hold, let $h\in U_\sigma$, $\mu=\mu_h$ and $a=1/(Sh)$, and write $H^{r}_{0}(\mu):=H^{r}(X)\cap L^{2}_{0}(\mu)$ for $r\ge0$.
(i) $\Kmu$ restricts to a bijection of $H^{\sigma}_{0}(\mu)$ onto $H^{\sigma+2s}_{0}(\mu)$, with inverse $\Kmu^{-1}\varphi=h^{-1}\,S^{-1}\left((Sh)\,\varphi\right)$.
(ii) For $r\in\{\sigma,\sigma+2s\}$ and $m\in\{1,2\}$, $\id-\Kmu^{m}$ restricts to a bijection of $H^{r}_{0}(\mu)$.
(iii) $\Theta_\mu$ restricts to a bijection of $H^{\sigma}_{0}(\mu)$ onto $H^{\sigma+2s}_{0}(\mu)$.
\end{lemma}

\begin{proof}
By (M2), $H^{\sigma}(X)$ and $H^{\sigma+2s}(X)$ are algebras that contain the inverses of their positive elements, so $h$ and $1/h$ lie in $H^{\sigma}(X)$, $Sh$ and $a$ in $H^{\sigma+2s}(X)$, and the product of a function in $H^{\sigma}(X)$ with one in $H^{\sigma+2s}(X)$ lies in $H^{\sigma}(X)$; by (M1), $S$ maps $H^{r}(X)$ onto $H^{r+2s}(X)$ for $0\le r\le\sigma$. Hence $\Kmu=M_aSM_h$ maps $H^{r}(X)$ into $H^{(\min(r,\sigma)+2s)}(X)$ for $0\le r\le\sigma+2s$, in particular $L^{2}(\mu_0)$ into $H^{2s}(X)$, $H^{\sigma}(X)$ into $H^{\sigma+2s}(X)$ and $H^{\sigma+2s}(X)$ into itself. Being self-adjoint on $L^{2}(\mu)$ with $\Kmu1=1$, it preserves $L^{2}_{0}(\mu)$.

(i) For $\varphi\in H^{\sigma+2s}(X)$, $(Sh)\varphi\in H^{\sigma+2s}(X)$, $S^{-1}((Sh)\varphi)\in H^{\sigma}(X)$ and
\[
\psi:=h^{-1}S^{-1}((Sh)\varphi)\in H^{\sigma}(X),
\]
with $\Kmu\psi=a\,S\left(S^{-1}((Sh)\varphi)\right)=a\,(Sh)\,\varphi=\varphi$. So $\Kmu$ maps $H^{\sigma}(X)$ onto $H^{\sigma+2s}(X)$, and it is injective on $L^{2}_{0}(\mu)$ by Lemma~\ref{lem:spectral}. If $\varphi$ has $\mu$-mean zero, so has $\psi$, since $\int_X\psi\,d\mu=\langle\psi,\Kmu1\rangle_{L^2(\mu)}=\int_X\Kmu\psi\,d\mu=0$.

(ii) $\Kmu^{m}$ maps $H^{r}_{0}(\mu)$ into itself for both values of $r$, so $\id-\Kmu^{m}$ is a bounded operator on $H^{r}_{0}(\mu)$, injective as on $L^{2}_{0}(\mu)$. For $\varphi\in H^{r}_{0}(\mu)$ let $\psi=(\id-\Kmu^{m})^{-1}\varphi\in L^{2}_{0}(\mu)$, with the inverse of Proposition~\ref{prop:operators}(iii). Then $\psi=\varphi+\Kmu^{m}\psi$, and by iteration $\psi=\sum_{j<n}\Kmu^{mj}\varphi+\Kmu^{mn}\psi$ for every $n$; each $\Kmu^{mj}\varphi$ lies in $H^{r}(X)$ and $\Kmu^{mn}\psi$ in $H^{\min(r,2mns)}(X)$, so $\psi\in H^{r}_{0}(\mu)$ if $2mns\ge r$.

(iii) follows from (i) and (ii) at the level $\sigma$, since $\Theta_\mu=\frac{\eps}{2}\Kmu(\id-\Kmu^{2})^{-1}$.
\end{proof}

% el rango del flat en el estrato, es lo que usa el Zusammenhang de la seccion 5
\begin{proposition}[Range of the flat map on the stratum]\label{prop:range}
Let Assumptions~\ref{assump:standing} and~\ref{assump:stratum} hold, let $h\in U_\sigma$, $\mu=\mu_h$ and $a=1/(Sh)$, with $H^{r}_{0}(\mu)$ as in Lemma~\ref{lem:sobolevK}.
(i) $T_\mu\mathcal P^{\sigma}(X)=\left\{\varphi\mu:\varphi\in H^{\sigma}_{0}(\mu)\right\}$, and a functional $w\mapsto\int_XF\,dw$ on it with $F\in\CX$ lies in the range of the flat map $\varphi\mu\mapsto g_\mu(\varphi\mu,\cdot)$ if and only if $F\in H^{\sigma+2s}(X)+\mathbb R$; it is then the image of $F^{\sharp}=\frac{2}{\eps}\left(\Kmu^{-1}-\Kmu\right)\Pi_0F\cdot\mu$, whose density with respect to $\mu$ lies in $H^{\sigma}_{0}(\mu)$.
(ii) $T_h$ is a bijection of $T_hU_\sigma$, and the range of $S\,T_h$ on $T_hU_\sigma$ is $\left\{w\in H^{\sigma+2s}(X):\langle h,w\rangle_0=0\right\}$.
\end{proposition}

\begin{proof}
(i) By Theorem~\ref{thm:scale}(i), $\mu=\rho\,\mathrm{vol}$ with $\rho\in H^{\sigma}(X)$ positive, and the tangent space consists of the measures $\psi\,\mathrm{vol}$ with $\psi\in H^{\sigma}(X)$ and $\int_X\psi\,d\mathrm{vol}=0$; such a measure is $\varphi\mu$ with $\varphi=\psi/\rho$, which lies in $H^{\sigma}(X)$ by (M2) and has $\mu$-mean zero, and conversely. By \eqref{eq:Theta}, $g_\mu(\varphi\mu,\chi\mu)=\int_X\Theta_\mu\varphi\,d(\chi\mu)$ for $\varphi,\chi\in H^{\sigma}_{0}(\mu)$, so the range of the flat map consists of the functionals $w\mapsto\int_XF\,dw$ with $F\in\Theta_\mu(H^{\sigma}_{0}(\mu))=H^{\sigma+2s}_{0}(\mu)$, by Lemma~\ref{lem:sobolevK}(iii). Two continuous functions define the same functional on $T_\mu\mathcal P^{\sigma}(X)$ if and only if they differ by a constant: tested against $\chi\mu$ with $\chi\in H^{\sigma}_{0}(\mu)$, a space dense in $L^{2}_{0}(\mu)$, their difference is $\mu$-a.e. constant, hence constant, since $\mu$ has full support. So the functional of $F$ lies in the range if and only if $\Pi_0F\in H^{\sigma+2s}_{0}(\mu)$, that is $F\in H^{\sigma+2s}(X)+\mathbb R$, and its preimage is $\Theta_\mu^{-1}\Pi_0F\cdot\mu=F^{\sharp}$, with density in $H^{\sigma}_{0}(\mu)$ by Lemma~\ref{lem:sobolevK}(iii).

(ii) $T_h=\id+2M_hRM_aS$ maps $T_hU_\sigma$ into itself: for $v\in T_hU_\sigma$, $aSv\in L^{2}_{0}(\mu)$ by Proposition~\ref{prop:stratum}(iii), $R(aSv)$ has $\mu$-mean zero, and $\langle T_hv,Sh\rangle_0=\langle v,Sh\rangle_0+2\int_XR(aSv)\,d\mu=0$. The second summand of $T_h$ is compact on $T_hU_\sigma$: $aSv\in H^{\sigma+2s}_{0}(\mu)$ by (M1) and (M2), so $v\mapsto R(aSv)$ is bounded into $H^{\sigma+2s}(X)$ by Lemma~\ref{lem:sobolevK}(ii) with $m=1$ and the open mapping theorem, the embedding of $H^{\sigma+2s}(X)$ into $H^{\sigma+s}(X)$ is compact, and multiplication by $h$ is bounded from $H^{\sigma+s}(X)$ into $H^{\sigma}(X)$. So $T_h$ is Fredholm of index zero on $T_hU_\sigma$ and injective, since $T_hv=0$ implies $g^{\sigma}_h(v,v)=0$ and then $v=0$ by Proposition~\ref{prop:stratumflat}(ii); hence it is bijective, and $S\,T_h(T_hU_\sigma)=S(T_hU_\sigma)=\{Sv:\langle v,Sh\rangle_0=0\}=\{w\in H^{\sigma+2s}(X):\langle S^{-1}w,Sh\rangle_0=0\}$ by (M1), and $\langle S^{-1}w,Sh\rangle_0=\langle w,h\rangle_0$ by the self-adjointness of $S$ on $L^{2}(\mu_0)$.
\end{proof}

\begin{remark}[Fisher-Rao against Sinkhorn on the stratum]\label{rem:frsinkhorn}
Let Assumptions~\ref{assump:standing} and~\ref{assump:stratum} hold, $h\in U_\sigma$ and $v\in T_hU_\sigma$, write $K=K_{\mu_h}$, and let $\psi_v$ be the density \eqref{eq:tangentialmass} of the tangent measure with respect to $\mu_h$. Both metrics are evaluated on it,
\[
g^{F}_h(v,v)=\|\psi_v\|_{L^2(\mu_h)}^{2},
\qquad
g^{\sigma}_h(v,v)=\left\langle\Theta_{\mu_h}\psi_v,\psi_v\right\rangle_{L^2(\mu_h)},
\]
the first by Remark~\ref{rem:ajls}, the second by Proposition~\ref{prop:stratumflat}(i). So the Fisher-Rao and the Sinkhorn metric are the forms $\langle\psi_v,\phi(K)\psi_v\rangle_{L^2(\mu_h)}$ of Lemma~\ref{lem:spectral} with the symbols $\phi_F(\lambda)=1$ and $\phi_S(\lambda)=\frac{\eps}{2}\,\frac{\lambda}{1-\lambda^{2}}$, the latter by \eqref{eq:Theta}, and since $\phi_S/\phi_F$ increases on $(0,q]$, Lemma~\ref{lem:spectral}(ii) yields
\begin{equation}\label{eq:fr-sinkhorn}
g^{\sigma}_h(v,v)\le\frac{\eps}{2}\,\frac{q}{1-q^{2}}\;g^{F}_h(v,v)
\qquad\text{for all }h\in U_\sigma\text{ and }v\in T_hU_\sigma,
\end{equation}
with a constant that depends only on $\eps$ and $\|c\|_\infty$. The inequality does not hold in the other direction: the functions $\psi_v=D\Phi_\sigma(h)v/\Phi_\sigma(h)$, $v\in T_hU_\sigma$, are exactly the functions in $H^{\sigma}(X)$ of $\mu_h$-mean zero, since $D\Phi_\sigma(h)$ maps $T_hU_\sigma$ onto the densities of mass zero (Theorem~\ref{thm:scale} and its proof), so they are dense in $L^2_0(\mu_h)$, and $\Theta_{\mu_h}$ is compact, with eigenvalues $\phi_S(\lambda_j)$ accumulating at zero. For a function on the stratum whose gradients with respect to both metrics exist at $h$, $\operatorname{grad}^{S}=\Theta_{\mu_h}^{-1}\operatorname{grad}^{F}$ as densities with respect to $\mu_h$: the component of the Fisher-Rao gradient along the eigenfunction of $K$ for the eigenvalue $\lambda$ is multiplied by $1/\phi_S(\lambda)=\frac{2}{\eps}(1-\lambda^{2})/\lambda$, an amplification of the rough directions that grows without bound as $\lambda\to0$.
\end{remark}

%% file: sec5.tex
\section{Third variation, self-duality and the Levi-Civita connection}\label{sec:tensors}

Eguchi's construction \citep{Eguchi_1983,Eguchi_1992} continues to third order: the third mixed derivatives of a divergence across its diagonal define the first-kind forms of a pair of torsion-free affine connections, dual to each other with respect to the metric, whose difference is a totally symmetric cubic tensor; for the Kullback-Leibler divergence they are the forms of the mixture and the exponential connection, and the cubic tensor is a multiple of the Amari-Chentsov tensor. Because $\Se$ is symmetric, its two forms coincide, with each other and with the first-kind Levi-Civita form of $g$, and the cubic tensor vanishes (Theorem~\ref{thm:collapse}). Since the metric is weak, the common form need not be represented by a tangent vector; we show that the form at a given base measure is represented exactly when it lies in the domain of the Riesz map (Remark~\ref{rem:sharp}), and on every smooth stratum we find such a representer, the Levi-Civita connection (Theorem~\ref{thm:levicivita}).

In Eguchi's convention the connection forms are the third mixed derivatives $-\partial_t\partial_s\partial_rD$ of a divergence $D$ on the diagonal, with two derivatives in one argument and one in the other; we scale them by the same factor (one half) as the metric \eqref{eq:gdef}, which keeps the duality identity \eqref{eq:dualitaet}, homogeneous of degree one in the triple, and Theorem~\ref{thm:collapse} does not depend on the normalization, since a constant multiple of $g$ has the same Levi-Civita connection.

In this section $\mu\in\PX$ is fixed, and the directions are those of Theorem~\ref{thm:hessian}, the vector space of the balanced measures with bounded density with respect to $\mu$,
\[
V_\mu=\left\{\xi\in\MzX:\ \xi\ll\mu,\ \frac{d\xi}{d\mu}\in L^{\infty}(\mu)\right\}.
\]
For $u,v,w\in V_\mu$ and $r,s,t$ small real parameters, the measures $\mu+tu+sv$ and $\mu+rw$ lie in $\PX$, since their densities with respect to $\mu$ stay positive and their mass is one; $V_{\mu+rw}=V_\mu$ for small $r$, since $d(\mu+rw)/d\mu$ is bounded above and below by positive constants. The connection forms $\Gamma_\mu,\Gamma^\ast_\mu:V_\mu\times V_\mu\times V_\mu\to\mathbb R$ are
\begin{equation}\label{eq:zusammenhangsformen}
\begin{aligned}
\Gamma_\mu(u,v;w)&=-\frac12\,\partial_t\partial_s\partial_r\,\OTe\left(\mu+tu+sv,\ \mu+rw\right)\Big|_{t=s=r=0},\\
\Gamma^\ast_\mu(u,v;w)&=-\frac12\,\partial_t\partial_s\partial_r\,\OTe\left(\mu+rw,\ \mu+tu+sv\right)\Big|_{t=s=r=0}.
\end{aligned}
\end{equation}

The function differentiated in \eqref{eq:zusammenhangsformen} is $\OTe$: its mixed derivatives are those of $\Se$, since the diagonal terms in \eqref{eq:selfkill} depend on one argument each. The derivatives are taken along straight lines in $\MX$, so $\Gamma_\mu$ and $\Gamma^\ast_\mu$ are connection forms relative to the affine structure of $\MX$, in which the mixture connection is flat, and they depend on that choice; their difference
\begin{equation}\label{eq:kubischertensor}
T_\mu(u,v,w)=\Gamma^\ast_\mu(u,v;w)-\Gamma_\mu(u,v;w)
\end{equation}
is a tensor, the cubic tensor of the divergence (for the Kullback-Leibler divergence a constant multiple of the Amari-Chentsov tensor).\footnote{Following \citet{AyJostLeSchwachhoefer2017}, we reserve the name Amari-Chentsov tensor for the canonical cubic tensor of a statistical model, $\int_X\psi_1\psi_2\psi_3\,d\mu$ for logarithmic derivatives $\psi_1,\psi_2,\psi_3$ of the density. On the smooth stratum this is the tensor of Remark~\ref{rem:ajls}, and it does not vanish identically. It is a different object from the cubic tensor $T$ of $\Se$, which vanishes by Theorem~\ref{thm:collapse}(i).} For vector fields $u,v,w$ that are constant in this affine structure, $\Gamma_\mu(u,v;w)$ plays the role of $g_\mu(\nabla_uv,w)$, the connection in first-kind form.

\begin{lemma}[{Smoothness along affine families; \citep[Theorem~4.3]{CarlierLaborde2020}}]\label{lem:smooth}
Let $\mu\in\PX$ and $u_1,\dots,u_p,\,w_1,\dots,w_q\in V_\mu$. For $t\in\mathbb R^{p}$ and $r\in\mathbb R^{q}$ set $\mu_t=\mu+\sum_i t_iu_i$ and $\nu_r=\mu+\sum_j r_jw_j$. Then there is a neighbourhood $U$ of the origin in $\mathbb R^{p+q}$ with $\mu_t,\nu_r\in\PX$ for $(t,r)\in U$. On $U$, both the potential class $(t,r)\mapsto\left[f_{\mu_t,\nu_r}\right]\in\CX/\mathbb R$ and the value $(t,r)\mapsto\OTe\left(\mu_t,\nu_r\right)$ are $C^{\infty}$.
\end{lemma}

In the notation of \citet[Theorem~4.3]{CarlierLaborde2020}, with $\mu$ as reference measure on both factors, the solution of the Schr\"odinger system for the marginal densities $\rho_t=d\mu_t/d\mu$ and $\varpi_r=d\nu_r/d\mu$ is $\left(f_{t,r}/\eps+\log\rho_t,\ g_{t,r}/\eps+\log\varpi_r\right)$, a $C^{\infty}$ function of $(\rho_t,\varpi_r)$ with values in $L^{\infty}(\mu)\times L^{\infty}(\mu)$. The densities are affine in $(t,r)$ and bounded below near the origin, and $T_\eps$ maps $L^{\infty}(\mu)$ smoothly into $\CX$, so the potentials are $C^{\infty}$ with values in $\CX$, and so is the value, by \eqref{eq:eot-value}.

By Lemma~\ref{lem:smooth} the derivatives in \eqref{eq:zusammenhangsformen} exist and commute; the forms $\Gamma_\mu$ and $\Gamma^\ast_\mu$ are linear in each argument, by the chain rule applied to families with several directions, and symmetric in the pair $(u,v)$, since exchanging $(t,u)$ and $(s,v)$ leaves $\mu+tu+sv$ unchanged. This is the torsion-free property, obtained directly from the definition; Proposition~\ref{prop:gammapot} expresses the same forms through derivatives of the potentials, whereby $u$ only pairs with the potential and $v$ varies its first marginal.

\begin{proposition}[Potential representation]\label{prop:gammapot}
For $u,v,w\in V_\mu$, let
\[
D^{2}f_\mu(v;w):=\partial_s\partial_r\,f_{\mu+sv,\,\mu+rw}\Big|_{s=r=0}\ \in\ \CX/\mathbb R
\]
be the mixed second derivative of the first potential, with one derivative in each marginal; it exists by Lemma~\ref{lem:smooth}. Then
\begin{equation}\label{eq:gammapot}
\Gamma_\mu(u,v;w)=-\frac12\left\langle u,\ D^{2}f_\mu(v;w)\right\rangle.
\end{equation}
The pairing does not depend on the representative, since $u$ is balanced.
\end{proposition}

\begin{proof}
Let $F(t,s,r)=\OTe\left(\mu+tu+sv,\,\mu+rw\right)$, which is $C^{\infty}$ near the origin by Lemma~\ref{lem:smooth}. Since $V_{\mu+sv}=V_\mu$, the envelope identity \eqref{eq:envelope} at the base point $\mu+sv$ yields $\partial_tF(0,s,r)=\left\langle u,\ f_{\mu+sv,\,\mu+rw}\right\rangle$, and the pairing with the balanced measure $u$, a bounded linear functional on $\CX/\mathbb R$, commutes with the derivatives in $(s,r)$ of the $C^{\infty}$ potential class. Hence $\partial_t\partial_s\partial_rF(0)=\left\langle u,D^{2}f_\mu(v;w)\right\rangle$, and \eqref{eq:zusammenhangsformen} yields \eqref{eq:gammapot}.
\end{proof}

Differentiating the Schr\"odinger system twice expresses $D^{2}f_\mu$ through the second derivatives of $T_\eps$, the operators $\Kmu$ and $\Hmu$ and the resolvent $\left(\id-\Kmu^{2}\right)^{-1}$ (Proposition~\ref{prop:Dsqf} in Appendix~\ref{app:thirdvar}); Theorem~\ref{thm:collapse} does not need this expression.

\begin{theorem}[Self-duality of the Sinkhorn divergence]\label{thm:collapse}
Let $\mu\in\PX$ and $u,v,w\in V_\mu$. Then:

(i) $\Gamma_\mu=\Gamma^\ast_\mu$; in particular the cubic tensor of $\Se$ vanishes identically, $T\equiv0$.

(ii) The common connection form is metric: with $g$ as in \eqref{eq:metrik},
\begin{equation}\label{eq:dualitaet}
\partial_w\,g_\mu(u,v)=\Gamma_\mu(w,u;v)+\Gamma_\mu(w,v;u),
\end{equation}
where the left side is the derivative at $0$ of $r\mapsto g_{\mu+rw}(u,v)$.

(iii) $\Gamma_\mu$ is symmetric in its first pair. Consequently it satisfies the Koszul formula
\begin{equation}\label{eq:koszulformel}
2\,\Gamma_\mu(u,v;w)=\partial_u\,g_\mu(v,w)+\partial_v\,g_\mu(u,w)-\partial_w\,g_\mu(u,v),
\end{equation}
and it is the first-kind Levi-Civita form of $g$ relative to the affine structure of $\MX$: the unique connection form on $V_\mu$ that is torsion-free and metric.
\end{theorem}

\begin{proof}
(i) The transposition $(x,y)\mapsto(y,x)$ of $X\times X$ maps $\Pi(\mu,\nu)$ onto $\Pi(\nu,\mu)$ and $\mu\otimes\nu$ to $\nu\otimes\mu$, and it preserves the transport cost, since $c$ is symmetric (Assumption~\ref{assump:standing}), and the relative entropy, which is unchanged when the same bijection is applied to both of its arguments; so $\OTe(\nu,\mu)=\OTe(\mu,\nu)$, the two third derivatives in \eqref{eq:zusammenhangsformen} coincide, $\Gamma^\ast_\mu=\Gamma_\mu$, and $T\equiv0$ by \eqref{eq:kubischertensor}.

(ii) Both sides of \eqref{eq:dualitaet} are derivatives of $G(t,s,r_1,r_2)=\OTe\left(\mu+tu+r_1w,\,\mu+sv+r_2w\right)$, which is $C^{\infty}$ near the origin by Lemma~\ref{lem:smooth}: by \eqref{eq:gF} and Theorem~\ref{thm:hessian}, extended to two directions by the polarization identity for quadratic forms at the base point $\mu+rw$, where $u,v\in V_{\mu+rw}=V_\mu$, $g_{\mu+rw}(u,v)=-\tfrac12\,\partial_t\partial_s\,G(t,s,r,r)\big|_{t=s=0}$, and the chain rule for $r\mapsto(r,r)$ yields
\[
\partial_w\,g_\mu(u,v)
=-\tfrac12\,\partial_t\partial_s\partial_{r_1}G\big|_{0}-\tfrac12\,\partial_t\partial_s\partial_{r_2}G\big|_{0}
=\Gamma_\mu(u,w;v)+\Gamma^\ast_\mu(v,w;u),
\]
the two terms being the derivatives \eqref{eq:zusammenhangsformen} with $u,w$ in the first argument and $v$ in the second, and with $v,w$ in the second and $u$ in the first. With (i) and the pair symmetry this is \eqref{eq:dualitaet}.

(iii) The pair symmetry was shown after Lemma~\ref{lem:smooth}. Writing \eqref{eq:dualitaet} for $\partial_u\,g_\mu(v,w)$, $\partial_v\,g_\mu(u,w)$ and $\partial_w\,g_\mu(u,v)$, adding the first two identities and subtracting the third, the pair symmetry cancels four of the six terms, and $\Gamma_\mu(u,v;w)+\Gamma_\mu(v,u;w)=2\,\Gamma_\mu(u,v;w)$ remains, which is \eqref{eq:koszulformel}. The same computation goes through for any trilinear form on $V_\mu$ that is symmetric in its first pair and satisfies \eqref{eq:dualitaet}, so it equals the right side of \eqref{eq:koszulformel}, which depends only on $g$.
\end{proof}

\begin{remark}[Self-duality of symmetric divergences]\label{rem:selfdual}
The proof of Theorem~\ref{thm:collapse} uses only the symmetry of the divergence, for (i), and the smoothness of its value along affine families, for (ii) and (iii): every symmetric divergence with these derivatives is self-dual in Eguchi's sense, with the first-kind Levi-Civita form of its metric as its third-order theory. In particular $\Se$ \textit{does not induce two distinct dual connections in the sense of Amari}. A dual pair can instead be obtained from the metric, following \citet{Ay2024Otto} and \citet{AySchwachhoefer2025Torsion}: in the regular setting of Remark~\ref{rem:rieszmaps}, the mixture connection and its dual with respect to $g$ are defined, and their mean is not the Levi-Civita connection of Theorem~\ref{thm:levicivita}, since $g$ is not a Hessian metric (Proposition~\ref{prop:nothessian}) \citep[Corollary~1]{AySchwachhoefer2025Torsion}. Adding to $\OTe$ a function of $\mu$ and a function of $\nu$ that are smooth along affine families changes neither the metric nor the connection forms, by \eqref{eq:selfkill}; an asymmetric third-order theory would therefore require an asymmetric transport problem, for instance an asymmetric cost or reference measure.
\end{remark}

By Theorem~\ref{thm:collapse}(ii), the first variation of the metric is $\partial_wg_\mu(u,v)=\Gamma_\mu(w,u;v)+\Gamma_\mu(w,v;u)$. If $g$ is the Hessian of a function in the affine structure of $\MX$, then $\partial_wg_\mu(u,v)$ is a third derivative of that function and is symmetric in all three directions; by the Koszul formula \eqref{eq:koszulformel} this holds if and only if $\Gamma_\mu(u,v;w)$ is symmetric in $(v,w)$, and by Proposition~\ref{prop:gammapot} if and only if $D^{2}f_\mu(v;w)-D^{2}f_\mu(w;v)$ is constant on the support of $\mu$ for all $v,w\in V_\mu$; when $\mu$ has full support, this holds if and only if $D^{2}f_\mu(v;w)=D^{2}f_\mu(w;v)$ in $\CX/\mathbb R$. This symmetry of the first potential between its two marginals is then a \textit{different one} from the self-duality $\Gamma_\mu=\Gamma^\ast_\mu$, which comes from the transposition and implies neither the symmetry nor its failure; Proposition~\ref{prop:nothessian} shows that it fails.

On the torus the formula of Appendix~\ref{app:thirdvar} for $D^{2}f_\mu$ (Proposition~\ref{prop:Dsqf}) becomes explicit. Let $X=\mathbb T^d=\mathbb R^d/\mathbb Z^d$ with a translation-invariant cost $c(x,y)=\phi(x-y)$, where $\phi$ is even, continuous and not constant, and let $\mu$ be the uniform measure. The translates of the self-potential solve \eqref{eq:fixpunkt} as well, since translations preserve $\mu$ and $c$, so by uniqueness (\autoref{subsec:sinkhorn}) the self-potential is constant. Then $\Kmu$ is the convolution with the probability density $e^{-\phi/\eps}/\int_Xe^{-\phi/\eps}\,dz$, and its eigenfunctions are the cosines $e_m(x)=\cos(2\pi m\cdot x)$, $m\in\mathbb Z^d$, with eigenvalues
\begin{equation}\label{eq:torus-eigen}
\lambda_m=\frac{\int_Xe^{-\phi(z)/\eps}\cos(2\pi m\cdot z)\,dz}{\int_Xe^{-\phi(z)/\eps}\,dz},
\end{equation}
which lie in $(-1,1)$ for $m\neq0$. The cosines are therefore the natural directions at $\mu$, and on them $\Kmu$, $\Hmu$ and the resolvents of Appendix~\ref{app:thirdvar} act diagonally, while the products in $\Xi_\mu$ of \eqref{eq:Xi} mix the modes.

\begin{proposition}[The Sinkhorn metric is not a Hessian metric]\label{prop:nothessian}
In this setting let $n\in\mathbb Z^d\setminus\{0\}$, and let $u$ and $v$ be the measures with densities $e_n$ and $e_{2n}$ with respect to $\mu$. Then, for every $\eps>0$,
\begin{equation}\label{eq:nichthessesch}
\Gamma_\mu(u,v;u)-\Gamma_\mu(u,u;v)=\frac{\eps}{8}\,\frac{\lambda_n\left(\lambda_n-\lambda_{2n}\right)}{\left(1-\lambda_n\right)^{2}\left(1+\lambda_n\right)\left(1+\lambda_{2n}\right)}.
\end{equation}
There is an $n$ for which the right side does not vanish, so $g$ is not a Hessian metric in the affine structure of $\MX$. For the unit circle in $\mathbb R^2$ with half the squared distance of the plane as cost, $c(x,y)=1-\cos(2\pi(x-y))$ on $\mathbb R/\mathbb Z$, the eigenvalues are $\lambda_m=I_m(1/\eps)/I_0(1/\eps)$, with $I_m$ the modified Bessel functions of the first kind; they are positive and strictly decreasing in $|m|$, and $n=1$ works for every $\eps$. The same holds for the Mat\'ern kernels of \autoref{sec:stratum}, with $\lambda_m=\left(\kappa^{2}/(\kappa^{2}+4\pi^{2}|m|^{2})\right)^{s}$.
\end{proposition}

\begin{proof}
By \eqref{eq:Kmu}, $\Hmu[\psi\mu]=\Kmu\psi$ for bounded $\psi$. In the notation of Appendix~\ref{app:thirdvar}, $\varphi_u=\lambda_n(1-\lambda_n^{2})^{-1}e_n$ and $\varphi_v=\lambda_{2n}(1-\lambda_{2n}^{2})^{-1}e_{2n}$, and every product in \eqref{eq:Xi} is a product of two of the cosines $e_n$ and $e_{2n}$, with $e_ne_{2n}=\frac12\left(e_n+e_{3n}\right)$ and $e_n^{2}=\frac12\left(1+e_{2n}\right)$. In \eqref{eq:gammaop} only the component along $e_n$ pairs with $u$. Collecting it in the six terms of $\Xi_\mu(v,u)$ and of $\Xi_\mu(u,v)$ yields
\[
\int_Xe_n\left(\Xi_\mu(v,u)-\Xi_\mu(u,v)\right)d\mu=\frac14\,\frac{\lambda_n\left(\lambda_{2n}-\lambda_n\right)}{\left(1-\lambda_n\right)\left(1+\lambda_{2n}\right)},
\]
and since $\left(\id-\Kmu^{2}\right)^{-1}$ acts on $e_n$ by $\left(1-\lambda_n^{2}\right)^{-1}$, \eqref{eq:gammaop} yields \eqref{eq:nichthessesch}. Since $\phi$ is not constant, neither is $e^{-\phi/\eps}$, so $\lambda_n\neq0$ for some $n\neq0$, and $\lambda_m\to0$ as $|m|\to\infty$ by the Riemann-Lebesgue lemma. If $\lambda_{2n}=\lambda_n$ held for every $n$ with $\lambda_n\neq0$, then $\lambda_{2^{j}n}=\lambda_n\neq0$ for all $j$, which is impossible. By the paragraph before the proposition, a Hessian metric has $\Gamma_\mu$ symmetric in its last two arguments. For the circle, $e^{\cos\theta/\eps}=\sum_{m\in\mathbb Z}I_m(1/\eps)\cos(m\theta)$, and $I_m(1/\eps)$ decreases strictly in $m\ge0$.
\end{proof}

For every $\eps$, the mixed second derivative of the first potential is not symmetric in its two marginals, $D^{2}f_\mu(v;w)\neq D^{2}f_\mu(w;v)$ for some $v,w\in V_\mu$. For a cost $\phi$ that does not depend on $\eps$, the asymmetry persists at order $1/\eps$ as $\eps\to\infty$, where the metric tends to the constant, hence Hessian, but possibly degenerate form of the maximum mean discrepancy (Remark~\ref{rem:epslimits}): the right side of \eqref{eq:nichthessesch} is $\hat\phi(n)\left(\hat\phi(n)-\hat\phi(2n)\right)/(8\eps)+O(\eps^{-2})$, with $\hat\phi(m)=\int_X\phi(z)\cos(2\pi m\cdot z)\,dz$ the Fourier coefficients of the cost, and the dyadic argument of the proof, applied to $\hat\phi$ in place of $\lambda$, yields an $n$ with $\hat\phi(n)\left(\hat\phi(n)-\hat\phi(2n)\right)\neq0$. For the Mat\'ern kernels the eigenvalues $\lambda_m$ do not depend on $\eps$, and the right side of \eqref{eq:nichthessesch} is linear in $\eps$.

\begin{remark}[First-kind and second-kind connection]\label{rem:sharp}
Theorem~\ref{thm:collapse} determines the connection through its first-kind form $\Gamma_\mu(u,v;\cdot)$, a linear form on $V_\mu$. By (i) and the envelope identity \eqref{eq:envelope}, applied in the first argument along $w$,
\[
\Gamma_\mu(u,v;w)=\Gamma^\ast_\mu(u,v;w)=\left\langle w,\ \gamma_{u,v}\right\rangle,
\qquad
\gamma_{u,v}:=-\frac12\,\partial_t\partial_s\,f_{\mu,\,\mu+tu+sv}\Big|_{t=s=0}\in\CX/\mathbb R,
\]
the second derivative of the first potential in its second marginal, which exists by Lemma~\ref{lem:smooth}. By Theorem~\ref{thm:gradients}, the form $\Gamma_\mu(u,v;\cdot)$ is represented by a balanced measure, a $\nabla_uv\in\MzX$ with $g_\mu(\nabla_uv,w)=\left\langle w,\gamma_{u,v}\right\rangle$ for all $w\in\MzX$, if and only if $\gamma_{u,v}\in\mathcal D_\mu$, and then $\nabla_uv=\gamma_{u,v}^\sharp=\operatorname{grad}_\mu E(\gamma_{u,v})$, as in weak Riemannian geometry, where the flat map need not be onto, as we have mentioned repeteadly. On every smooth stratum this \textit{is} however the case, with $\gamma_{u,v}^\sharp$ in the tangent space of the stratum and analytic in $(\mu,u,v)$: the Levi-Civita connection of Theorem~\ref{thm:levicivita}.
\end{remark}

On the smooth stratum the representation question of Remark~\ref{rem:sharp} has thus a complete answer. By Proposition~\ref{prop:range}(i) the form $\Gamma_\mu(u,v;\cdot)=\langle\cdot,\gamma_{u,v}\rangle$ is represented on $T_\mu\mathcal P^{\sigma}(X)$ exactly when $\gamma_{u,v}\in H^{\sigma+2s}(X)+\mathbb R$, and $\gamma_{u,v}$ is a second derivative of the potential $f_{\mu,\nu}$ in $\nu$ along densities in $H^{\sigma}(X)$; so the question reduces to the regularity of the potentials of two marginals with densities in $H^{\sigma}(X)$, the two-marginal form of Corollary~\ref{cor:analytic}. Let $\mu=\rho\mu_0$ and $\nu=\varrho\mu_0$ with $\rho,\varrho\in H^{\sigma}(X)$ positive and of mass one, and put $a_{\mu,\nu}:=e^{f_{\mu,\nu}/\eps}$ and $b_{\mu,\nu}:=e^{g_{\mu,\nu}/\eps}$. Exponentiated, the Schr\"odinger system \eqref{eq:schroedingersystem} is
\begin{equation}\label{eq:schrodinger-exp}
a_{\mu,\nu}\,S\left(b_{\mu,\nu}\,\varrho\right)=1,
\qquad
b_{\mu,\nu}\,S\left(a_{\mu,\nu}\,\rho\right)=1,
\end{equation}
with $S$ the operator \eqref{eq:Sop}, and its solutions are the pairs $\left(\lambda a_{\mu,\nu},\,b_{\mu,\nu}/\lambda\right)$, $\lambda>0$ (\autoref{subsec:duality}). For $\nu=\mu$ the pair is $(\amu,\amu)$.

% para que la forma de Levi-Civita este representada hace falta gamma_{u,v} en H^{sigma+2s}, o sea la regularidad de los potenciales con dos marginales distintas
\begin{lemma}[Sobolev regularity and analytic dependence of the potentials]\label{lem:potanalytic}
Let Assumptions~\ref{assump:standing} and~\ref{assump:stratum} hold, and let $\mu,\nu$ be as above.
(i) $a_{\mu,\nu}$ and $b_{\mu,\nu}$ lie in $H^{\sigma+2s}(X)$.
(ii) The operators $L:=M_{a_{\mu,\nu}}SM_{b_{\mu,\nu}\varrho}$ and $L':=M_{b_{\mu,\nu}}SM_{a_{\mu,\nu}\rho}$ satisfy $L1=1$ and $L'1=1$, $L'$ is the adjoint of $L$ between $L^{2}(\nu)$ and $L^{2}(\mu)$, and $\|LL'\|\le(1-\eta)^{2}$ on $L^{2}_{0}(\mu)$ with $\eta:=e^{-\|c\|_\infty/\eps}\min a_{\mu,\nu}\min b_{\mu,\nu}>0$; and $\id-LL'$ restricts to a bijection of $H^{\sigma+2s}(X)\cap L^{2}_{0}(\mu)$.
(iii) The potential normalized by $\int_X\log a_{\mu,\nu}\,d\mu_0=0$ is a real-analytic function of $(\rho,\varrho)$ on the open set of positive mass-one densities in $H^{\sigma}(X)\times H^{\sigma}(X)$, with values in $H^{\sigma+2s}(X)$; so is $b_{\mu,\nu}=1/S(a_{\mu,\nu}\rho)$.
(iv) For $\mu\in\mathcal P^{\sigma}(X)$ and $u=\psi\mu_0$, $v=\chi\mu_0$ in $T_\mu\mathcal P^{\sigma}(X)$, the function $\gamma_{u,v}$ of Remark~\ref{rem:sharp} lies in $H^{\sigma+2s}(X)+\mathbb R$, and $(\rho,\psi,\chi)\mapsto\gamma_{u,v}$ is real-analytic into $H^{\sigma+2s}(X)/\mathbb R$ and bilinear in $(\psi,\chi)$.
\end{lemma}

\begin{proof}
Write $a=a_{\mu,\nu}$ and $b=b_{\mu,\nu}$ in this proof. (i) The potentials are continuous (\autoref{subsec:duality}), so $b\varrho\in L^{2}(\mu_0)$, $S(b\varrho)\in H^{2s}(X)$ by (M1) with $r=0$, and $S(b\varrho)\ge e^{-\|c\|_\infty/\eps}\min(b\varrho)>0$; by (M2) at the level $2s>d/2$, $a=1/S(b\varrho)\in H^{2s}(X)$. Then $a\rho\in H^{\min(2s,\sigma)}(X)$, $S(a\rho)\in H^{\min(2s,\sigma)+2s}(X)$, and $b=1/S(a\rho)$ lies in that space; returning to $a$ raises its regularity by $2s$ again, and so on, until $H^{\sigma+2s}(X)$ is reached, where the regularity of $\rho$ and $\varrho$ stops the iteration.

(ii) $L1=aS(b\varrho)=1$ and $L'1=1$ are the two equations of \eqref{eq:schrodinger-exp}, and $\langle L\varphi,\chi\rangle_{L^2(\mu)}=\int_X\int_X a(x)k_c(x,y)b(y)\varphi(y)\chi(x)\,d\nu(y)\,d\mu(x)=\langle\varphi,L'\chi\rangle_{L^2(\nu)}$. The kernel of $L$ with respect to $\nu$ is $\ell(x,y)=a(x)k_c(x,y)b(y)\ge\eta$, with $\int_X\ell(x,y)\,d\mu(x)=b(y)S(a\rho)(y)=1$; so $m:=(\ell-\eta)/(1-\eta)$ is the kernel of a Markov operator $M$ from $L^{2}(\nu)$ to $L^{2}(\mu)$ with $\int_Xm(x,y)\,d\mu(x)=1$, and Jensen's inequality yields $\|M\varphi\|_{L^2(\mu)}\le\|\varphi\|_{L^2(\nu)}$. For $\varphi\in L^{2}_{0}(\nu)$, $L\varphi=(1-\eta)M\varphi$, so $\|L\varphi\|_{L^2(\mu)}\le(1-\eta)\|\varphi\|_{L^2(\nu)}$, and symmetrically for $L'$; both map functions of mean zero to functions of mean zero, since $L'1=1$ and $L1=1$. Hence $\|LL'\|\le(1-\eta)^{2}$ on $L^{2}_{0}(\mu)$, $LL'$ is self-adjoint on $L^{2}(\mu)$, and $\id-LL'$ is invertible on $L^{2}_{0}(\mu)$. By (i), (M1) and (M2), $L$ and $L'$ map $H^{r}(X)$ into $H^{\min(r,\sigma)+2s}(X)$ for $0\le r\le\sigma+2s$, so $\id-LL'$ maps $H^{\sigma+2s}(X)\cap L^{2}_{0}(\mu)$ into itself, injectively; and for $\varphi$ in this space the solution $\alpha=(\id-LL')^{-1}\varphi\in L^{2}_{0}(\mu)$ satisfies $\alpha=\sum_{j<n}(LL')^{j}\varphi+(LL')^{n}\alpha$, which lies in $H^{\sigma+2s}(X)$ if $4ns\ge\sigma+2s$.

% ojo con la normalizacion int log a dmu_0 = 0: sin ella D_a Z tiene Kern, las constantes
(iii) For positive $a\in H^{\sigma+2s}(X)$ and positive mass-one $\rho,\varrho\in H^{\sigma}(X)$ let
\[
\hat{\mathcal Z}(a;\rho,\varrho):=a\,S\left(\varrho/S(a\rho)\right)-1+\left(\int_X\log a\,d\mu_0\right)\mathbf 1 ,
\]
real-analytic with values in $H^{\sigma+2s}(X)$ by (M1), (M2) and the analyticity of $\log$ on the positive elements. The first two terms integrate to zero against $\mu$, by Fubini's theorem and the symmetry of $k_c$, so $\hat{\mathcal Z}(a;\rho,\varrho)=0$ if and only if $(a,1/S(a\rho))$ solves \eqref{eq:schrodinger-exp} and $\int_X\log a\,d\mu_0=0$, that is, $a$ is the normalized $a_{\mu,\nu}$. At a solution, with $\alpha=a\alpha'$ and $L,L'$ as in (ii), \eqref{eq:schrodinger-exp} yields
\[
D_a\hat{\mathcal Z}(a;\rho,\varrho)[a\alpha']=(\id-LL')\alpha'+\left(\int_X\alpha'\,d\mu_0\right)\mathbf 1 ,
\]
an isomorphism of $H^{\sigma+2s}(X)$: by (ii), $\int_X(\id-LL')\alpha'\,d\mu=0$, so integration against $\mu$ determines $\int_X\alpha'\,d\mu_0$, and in the remaining equation $(\id-LL')\alpha'=\varphi-\left(\int_X\varphi\,d\mu\right)\mathbf 1$ the operator $\id-LL'$ has the constants as kernel and is a bijection of $H^{\sigma+2s}(X)\cap L^{2}_{0}(\mu)$ by (ii), the constant being fixed by the integral condition. The implicit function theorem for real-analytic maps between Banach spaces \citep[Chapter~X, Section~2]{Dieudonne1960} then yields $a_{\mu,\nu}$ as a real-analytic function of $(\rho,\varrho)$ near every point, with values in $H^{\sigma+2s}(X)$, and $b_{\mu,\nu}=1/S(a_{\mu,\nu}\rho)$ with it.

(iv) For $\mu\in\mathcal P^{\sigma}(X)$ the density $\rho$ is positive and lies in $H^{\sigma}(X)$ (Theorem~\ref{thm:scale}(i)), and $\rho+t\psi+s\chi$ is a positive mass-one density in $H^{\sigma}(X)$ for small $t,s$, since $\int_X\psi\,d\mu_0=\int_X\chi\,d\mu_0=0$. By (iii), $(t,s)\mapsto f_{\mu,\mu+tu+sv}=\eps\log a_{\mu,\mu+tu+sv}$ is real-analytic into $H^{\sigma+2s}(X)$ for the normalization of (iii), so $\gamma_{u,v}=-\frac12\partial_t\partial_sf_{\mu,\mu+tu+sv}|_{t=s=0}$ lies in $H^{\sigma+2s}(X)$; the normalization of Remark~\ref{rem:sharp} differs from it by a constant depending on $(t,s)$, whose mixed derivative is a constant. The second derivative of an analytic map along an affine family is bilinear in the directions and analytic in the base point.
\end{proof}

% aqui se construye el Zusammenhang: nabla_u v es el sharp de gamma_{u,v}, que existe en cada estrato por prop:range(i) y lem:potanalytic(iv)
\begin{theorem}[Levi-Civita connection on the smooth stratum]\label{thm:levicivita}
Let Assumptions~\ref{assump:standing} and~\ref{assump:stratum} hold. For $\mu\in\mathcal P^{\sigma}(X)$ and $u,v\in T_\mu\mathcal P^{\sigma}(X)$ define
\begin{equation}\label{eq:zusammenhang}
\nabla_uv:=\gamma_{u,v}^{\sharp}=\frac{2}{\eps}\left(\Kmu^{-1}-\Kmu\right)\Pi_0\gamma_{u,v}\cdot\mu .
\end{equation}
(i) $\nabla_uv\in T_\mu\mathcal P^{\sigma}(X)$; it is the unique balanced measure with $g_\mu(\nabla_uv,w)=\langle w,\gamma_{u,v}\rangle$ for all $w\in\MzX$, and $g_\mu(\nabla_uv,w)=\Gamma_\mu(u,v;w)$ for $w\in V_\mu$.
(ii) In the density chart of Theorem~\ref{thm:scale}, $(\mu,u,v)\mapsto\nabla_uv$ is real-analytic and bilinear in $(u,v)$.
(iii) $\nabla$ is the Levi-Civita connection of the weak Riemannian metric $g^{\sigma}$ on $\mathcal P^{\sigma}(X)$: torsion-free, metric, and unique.
(iv) If $\sigma'>\sigma$ and Assumption~\ref{assump:stratum} holds for $\sigma'$, the connection of $\mathcal P^{\sigma'}(X)$ is the restriction of that of $\mathcal P^{\sigma}(X)$; on the torus with the Mat\'ern kernel, the common restriction to $\mathcal P^{\infty}_{+}(X)$ is the Levi-Civita connection of \citet[Theorem~1]{AySchwachhoefer2025Torsion}.
\end{theorem}

\begin{proof}
(i) By Lemma~\ref{lem:potanalytic}(iv), $\Pi_0\gamma_{u,v}\in H^{\sigma+2s}(X)\cap L^{2}_{0}(\mu)$, so by Proposition~\ref{prop:range}(i) the functional $w\mapsto\langle w,\gamma_{u,v}\rangle$ on $T_\mu\mathcal P^{\sigma}(X)$ is the image under the flat map of $\gamma_{u,v}^{\sharp}$, whose density lies in $H^{\sigma}(X)\cap L^{2}_{0}(\mu)$; and $\langle w,\gamma_{u,v}\rangle=\Gamma_\mu(u,v;w)$ for $w\in V_\mu$ by Remark~\ref{rem:sharp}. By Lemma~\ref{lem:sobolevK}(i), $\Pi_0\gamma_{u,v}=\Kmu\varphi$ with $\varphi\in H^{\sigma}(X)\cap L^{2}_{0}(\mu)$, so $(\id-\Kmu^{2})\Pi_0\gamma_{u,v}=\Hmu\left[(\id-\Kmu^{2})\varphi\,\mu\right]$ and $\gamma_{u,v}\in\mathcal D_\mu$; Theorem~\ref{thm:gradients}(i) then yields the uniqueness among all balanced measures.

(ii) $\gamma_{u,v}$ is analytic and bilinear by Lemma~\ref{lem:potanalytic}(iv); $\Pi_0$ is affine in $\rho$; and $\Kmu^{-1}=M_{1/h}S^{-1}M_{Sh}$ (Lemma~\ref{lem:sobolevK}(i)) and $\Kmu$ depend analytically on $\rho$ through $h=\Phi_\sigma^{-1}(\rho)$, by Theorem~\ref{thm:scale}(i), (M2) and Corollary~\ref{cor:analytic}.

(iii) The density chart is an open subset of an affine hyperplane of $H^{\sigma}(X)$, so by \citet[Section~2.2]{BauerBruverisMichor2014}, (i) and (ii) define a covariant derivative, $\nabla_UV=DV[U]+\nabla_{U(\mu)}V(\mu)$ for vector fields $U,V$ in the chart, whose first-kind form on constant fields is $\Gamma_\mu$. It is torsion-free because $\Gamma_\mu$ is symmetric in its first pair and metric because of \eqref{eq:dualitaet}, both by Theorem~\ref{thm:collapse}; and a torsion-free metric connection of $g^{\sigma}$ has the first-kind form \eqref{eq:koszulformel}, so it is unique, the flat map being injective (Proposition~\ref{prop:schwach}).

(iv) The right side of \eqref{eq:zusammenhang} does not involve $\sigma$, and by Theorem~\ref{thm:scale}(ii) a measure of $\mathcal P^{\sigma'}(X)$ has the same chart on both levels, with $T_\mu\mathcal P^{\sigma'}(X)\subset T_\mu\mathcal P^{\sigma}(X)$. On $\mathcal P^{\infty}_{+}(X)$, the restriction is a torsion-free metric connection for the metric, which is regular in the sense of \citet{AySchwachhoefer2025Torsion} by Remark~\ref{rem:rieszmaps}, and their Theorem~1 and the uniqueness in (iii) identify it with their Levi-Civita connection.
\end{proof}

%% file: sec6.tex
\section{Discussion and open directions}\label{sec:discussion}

Beside the Fisher-Rao geometry of the Kullback-Leibler divergence and the Otto geometry of the Wasserstein distance, the Sinkhorn divergence determines a new information geometry. By the gradient characterization of Theorem~\ref{thm:gradients}, its metric is of the same kind as the Fisher-Rao and the Otto metric: each of the three is the Fisher-Rao metric twisted by an operator $\Theta_\mu$ in the sense of \citet{AySchwachhoefer2025Torsion}, and its cotangent form is built from a process reversible for $\mu$ (Remark~\ref{rem:rieszmaps}). Since the Sinkhorn operator $\Theta_\mu$ is compact, on the stratum the Sinkhorn metric is bounded by a multiple of the Fisher-Rao metric, and the inequality does not hold in the other direction (Remark~\ref{rem:frsinkhorn}). Like the Sinkhorn divergence, which interpolates between optimal transport and the maximum mean discrepancy, the metric tensor has two limits in $\eps$ (Remark~\ref{rem:epslimits}). For a fixed cost and $\eps\to\infty$, $\Theta_\mu$ converges in operator norm on $L^2_0(\mu)$, and the metric tensor tends to the constant form of the maximum mean discrepancy with kernel $-c/2$; this form may be degenerate, and in that case it does not have a Riesz map. On the torus with the periodized Gaussian kernel and for smooth positive densities, the metric tends formally to the Otto metric as $\eps\to0$, and the sphere coordinate to the Hellinger embedding of the Fisher-Rao geometry. The limit is formal because the Laplace expansion of $\Kmu$ does not extend to the domain of the Riesz map; at the uniform measure it holds exactly on trigonometric polynomials. For finitely supported measures with fixed positive weights, whose distinct atoms vary smoothly in a compact subset of $\mathbb R^{d}$ with the quadratic cost, \citet{MordantXu2026} prove the convergence of the Sinkhorn metric and its derivatives of every fixed order, as $\eps\to0$ and uniformly on compact sets of atom configurations, to those of the Wasserstein metric of the atoms, and they explain why their argument does not extend to measures with a density; for general densities, the convergence of the metric tensor, of the Riesz map and of the sphere coordinate to their Otto and Fisher-Rao counterparts as $\eps\to0$ remains to be proved.

A weak Riemannian metric need not have a Levi-Civita connection or a geodesic flow, since its flat map is not onto. The Sinkhorn metric has a Levi-Civita connection on every smooth stratum (Theorem~\ref{thm:levicivita}), and in the density chart its geodesic equation is thus an ordinary differential equation with real-analytic right side, so local geodesics and the exponential map exist on each stratum. In current ongoing work, we derive the geodesic equation in the sphere coordinate and study the regularity of geodesics along the scale of strata, their relation to the minimizing paths of the induced distance of \citet[Section~5]{lavenant_luckhardt_mordant_et_al_riemannian_2025}, and the curvature of the Sinkhorn metric. Every object of this paper on $\PX$ is built only from the cost and is therefore invariant under the isometries of the cost, the homeomorphisms $\zeta$ of $X$ with $c(\zeta x,\zeta y)=c(x,y)$ for all $x,y\in X$, and the smooth stratum with its metric and its Levi-Civita connection under those cost isometries that are smooth diffeomorphisms, whereas the Fisher-Rao metric is invariant under \textit{every} homeomorphism of the sample space; a characterization of Chentsov type would therefore need axioms that relate the metric to the cost, for instance naturality under cost-preserving maps between pairs of sample space and cost, and compatibility in $\eps$.

In further ongoing work we develop the statistical side of the geometry, beginning with the gradient flow of the Bayesian free energy in the Sinkhorn metric, the geometry in which \citet{HardionLavenant2025} obtain gradient flows of potential energies from the minimizing-movement scheme with $\Se$ in place of the squared Wasserstein distance; at measures with a positive continuous density, the free energy is finite, and its gradient exists by Theorem~\ref{thm:gradients} wherever its first variation lies in $\mathcal D_\nu$. Estimation from samples concerns the gradients of mean values of continuous observables at empirical measures and is an inverse problem: the Riesz map amplifies the directions of small eigenvalues of $\Kmu$, which decay fast for a smooth kernel, by the factors $1/\phi_S(\lambda)$ of Remark~\ref{rem:frsinkhorn}, so the estimation of Sinkhorn gradients is unstable in those directions and requires for regularization, whereas the Riesz map of the Fisher-Rao metric is the identity on centered functions; the resolvent $(\id-\Kmu^{2})^{-1}$ of the Riesz map also appears in the central limit theorem of related work of \citet{harchaoui_liu_pal_asymptotics_2024} for random matchings weighted by a Gibbs measure.

%% file: app_proofs.tex
\iflong\else
\section{\texorpdfstring{Proofs for Section~\ref{sec:stratum}}{Proofs for Section 3}}\label{app:proofs-hilbert}

\begin{proof}[Proof of Lemma~\ref{lem:douglas}]
\PrfDouglas
\end{proof}

\begin{proof}[Proof of Proposition~\ref{prop:stratum}]
\PrfStratum
\end{proof}

\begin{proof}[Proof of Corollary~\ref{cor:analytic}]
\PrfAnalytic
\end{proof}

\begin{proof}[Proof of Theorem~\ref{thm:scale}]
\PrfScale
\end{proof}
\fi

%% file: app_thirdvar.tex
\section{An operator formula for the connection form}\label{app:thirdvar}

This appendix computes the operator form of the mixed second derivative $D^{2}f_\mu$ of the first Schr\"odinger potential, through which Proposition~\ref{prop:gammapot} expresses the connection form and on which the proof of Proposition~\ref{prop:nothessian} rests. Throughout, $\mu\in\PX$ is fixed, directions are taken in the vertical class $V_\mu$ of \autoref{sec:tensors}, the operators $\Kmu$ and $\Hmu$ are those of \autoref{subsec:operators}, every resolvent is taken in the resolvent convention, and we abbreviate
\[
\varphi_v:=\left(\id-\Kmu^{2}\right)^{-1}\Hmu[v]\ \in\CX
\qquad\text{for }v\in V_\mu,
\]
well defined and $\mu$-mean-zero since $\Hmu[v]$ is mean-zero for balanced $v$ by \eqref{eq:selfnorm}. For $\psi\in\CX$ and $\tau\in V_\mu$ we write $\psi\,\tau$ for the measure $d\left(\psi\tau\right)=\psi\,d\tau$, and products of functions are pointwise.

The following convention makes every displayed derivative an element of $\CX$ rather than an equivalence class. Among the smooth solution pairs obtained from Lemma~\ref{lem:smooth} for the two-parameter family $\left(\mu+sv,\,\mu+rw\right)$ we fix the gauge
\[
\left\langle f_{s,r},\,\mu\right\rangle=\left\langle\fmu,\,\mu\right\rangle,
\qquad
g_{s,r}:=T_\eps\left(f_{s,r},\,\mu+sv\right),
\]
so that $\left(f_{0,0},g_{0,0}\right)=\left(\fmu,\fmu\right)$, the pair solves the Schr\"odinger system \eqref{eq:schroedingersystem} exactly, and all parameter derivatives of $f_{s,r}$ at the origin are $\mu$-mean-zero functions.

The system is differentiated through the second derivatives of the entropic $c$-transform at the diagonal, which are expressed by the \emph{conditional covariance} of the self-kernel,
\begin{equation}\label{eq:covk}
\operatorname{Cov}_\mu\left(\psi_1,\psi_2\right):=\Kmu(\psi_1\psi_2)-(\Kmu\psi_1)\cdot(\Kmu\psi_2)\ \in\CX,
\end{equation}
the covariance of $\psi_1,\psi_2$ under the Markov kernel $k_\mu(\cdot,y)\,d\mu$, evaluated at $y$.

\begin{lemma}[Derivatives of the entropic $c$-transform at the diagonal]\label{lem:Tsecond}
For $\psi,\psi_1,\psi_2\in\CX$ and $\tau,\tau_1,\tau_2\in V_\mu$, the map $T_\eps$ of \eqref{eq:Teps} is twice differentiable at $\left(\fmu,\mu\right)$ along function and vertical-measure directions, with first derivatives \citep[Lemma~3.13]{lavenant_luckhardt_mordant_et_al_riemannian_2025}
\begin{equation}\label{eq:DTeps}
D_{\varphi}T_\eps[\psi]=-\Kmu\psi,
\qquad
D_{\rho}T_\eps[\tau]=-\eps\,\Hmu[\tau],
\end{equation}
and second derivatives
\begin{gather*}
D^{2}_{\varphi\varphi}T_\eps\left[\psi_1,\psi_2\right]=-\frac{1}{\eps}\operatorname{Cov}_\mu\left(\psi_1,\psi_2\right),
\qquad
D^{2}_{\rho\rho}T_\eps\left[\tau_1,\tau_2\right]=\eps\,\Hmu\left[\tau_1\right]\cdot\Hmu\left[\tau_2\right],\\
D^{2}_{\varphi\rho}T_\eps\left[\psi;\tau\right]=-\left(\Hmu\left[\psi\,\tau\right]-(\Kmu\psi)\cdot\Hmu[\tau]\right).
\end{gather*}
\end{lemma}

\begin{proof}
Write $N(\varphi,\rho)=\int_X e^{\left(\varphi(x)-c(x,\cdot)\right)/\eps}\,d\rho(x)$, so that $T_\eps=-\eps\log N$; the denominator at the diagonal is $N=e^{-\fmu/\eps}$ by \eqref{eq:fixpunkt}, and each normalized derivative of $N$ is one of the tilted averages: $D_\varphi N[\psi]/N=\eps^{-1}\Kmu\psi$, $D_\rho N[\tau]/N=\Hmu[\tau]$, $D^{2}_{\varphi\varphi}N\left[\psi_1,\psi_2\right]/N=\eps^{-2}\Kmu(\psi_1\psi_2)$, $D^{2}_{\varphi\rho}N[\psi;\tau]/N=\eps^{-1}\Hmu[\psi\,\tau]$, $D^{2}_{\rho\rho}N=0$. The chain rule $DT_\eps=-\eps\,DN/N$ and $D^{2}T_\eps=-\eps\left(D^{2}N/N-\left(DN/N\right)\otimes\left(DN/N\right)\right)$ assembles the displays; differentiability in the required joint sense is contained in the smoothness of the composed map established in the proof of Lemma~\ref{lem:smooth}. All three forms are invariant under constant shifts of $\varphi$, consistent with the gauge freedom.
\end{proof}

Differentiating the two equations of the system once at the origin, with \eqref{eq:DTeps}, yields $\partial_sf=-\Kmu\partial_sg$ and $\partial_sg=-\Kmu\partial_sf-\eps\Hmu[v]$, and the same pair in $r$ with the roles of $f$ and $g$ exchanged; solving with Proposition~\ref{prop:operators}(iii) yields, in the present notation,
\begin{equation}\label{eq:firstorder}
\begin{aligned}
\partial_s f\big|_{0}&=\eps\,\Kmu\varphi_v,
&\qquad
\partial_s g\big|_{0}&=-\eps\,\varphi_v,\\
\partial_r f\big|_{0}&=-\eps\,\varphi_w,
&\qquad
\partial_r g\big|_{0}&=\eps\,\Kmu\varphi_w,
\end{aligned}
\end{equation}
the derivatives of the potentials in their marginals of \citet[Proposition~3.14]{lavenant_luckhardt_mordant_et_al_riemannian_2025}.

\begin{proposition}[Mixed second derivative of the potential]\label{prop:Dsqf}
For $v,w\in V_\mu$, with the gauge above,
\begin{equation}\label{eq:Dsqf}
D^{2}f_\mu(v;w)=\eps\left(\id-\Kmu^{2}\right)^{-1}\Xi_\mu(v,w),
\end{equation}
where
\begin{equation}\label{eq:Xi}
\begin{aligned}
\Xi_\mu(v,w):={}&\operatorname{Cov}_\mu\left(\varphi_v,\Kmu\varphi_w\right)-\Kmu\operatorname{Cov}_\mu\left(\Kmu\varphi_v,\varphi_w\right)
+\Hmu\left[\varphi_v\,w\right]-(\Kmu\varphi_v)\cdot\Hmu[w]\\
&-\Kmu\Hmu\left[\varphi_w\,v\right]+\Kmu\left((\Kmu\varphi_w)\cdot\Hmu[v]\right).
\end{aligned}
\end{equation}
\end{proposition}

\begin{proof}
Differentiate $f_{s,r}=T_\eps\left(g_{s,r},\,\mu+rw\right)$ in $r$, then in $s$, and $g_{s,r}=T_\eps\left(f_{s,r},\,\mu+sv\right)$ in $s$, then in $r$, all at the origin; the base point varies in the function slot only, in the first equation through $g$, in the second through $f$, so Lemma~\ref{lem:Tsecond} yields, with the first formula of \eqref{eq:DTeps} and writing $f_{sr}=\partial_s\partial_rf\big|_{0}$, $g_{sr}=\partial_s\partial_rg\big|_{0}$,
\begin{equation}\label{eq:mixedsys}
\begin{aligned}
f_{sr}+\Kmu\,g_{sr}&=-\frac1\eps\operatorname{Cov}_\mu\left(\partial_sg,\partial_rg\right)
-\Hmu\left[\partial_sg\cdot w\right]+(\Kmu\partial_sg)\cdot\Hmu[w],\\
g_{sr}+\Kmu\,f_{sr}&=-\frac1\eps\operatorname{Cov}_\mu\left(\partial_sf,\partial_rf\right)
-\Hmu\left[\partial_rf\cdot v\right]+(\Kmu\partial_rf)\cdot\Hmu[v],
\end{aligned}
\end{equation}
a $2\times2$ linear system for $\left(f_{sr},g_{sr}\right)$ of the same shape as the first-order system before \eqref{eq:firstorder}. Eliminating $g_{sr}$ yields $\left(\id-\Kmu^{2}\right)f_{sr}=\eps\,\Xi_\mu(v,w)$ after the substitution of \eqref{eq:firstorder}, the six terms of \eqref{eq:Xi} arising in the displayed order from the two covariance terms and the four measure-slot terms. The right side is automatically $\mu$-mean-zero, being in the range of $\id-\Kmu^{2}$, and $f_{sr}$ is mean-zero by the gauge, so the inversion on the mean-zero subspace, Proposition~\ref{prop:operators}(iii) in the resolvent convention, yields \eqref{eq:Dsqf}; the companion $g_{sr}$ follows from the second equation of \eqref{eq:mixedsys}.
\end{proof}

By Proposition~\ref{prop:gammapot}, the connection form of Theorem~\ref{thm:collapse} becomes explicit:
\begin{equation}\label{eq:gammaop}
\Gamma_\mu(u,v;w)=-\frac{\eps}{2}\left\langle u,\ \left(\id-\Kmu^{2}\right)^{-1}\Xi_\mu(v,w)\right\rangle.
\end{equation}
The symmetry of the right side under the exchange of $u$ and $v$, guaranteed by Theorem~\ref{thm:collapse}(iii), is not visible in \eqref{eq:Xi} and constitutes a nontrivial family of operator identities.